\documentclass[10pt]{amsart}
\usepackage{amsmath,amsfonts,amsthm,amssymb,eufrak,comment}
\usepackage{amscd}
\usepackage{times,fancyhdr}
\usepackage{array}
\usepackage{xy}
\usepackage{epsfig}
\usepackage{graphicx}
\usepackage{color}
\usepackage[title]{appendix}

\newcommand{\re}{{\text{\rm re}}}
\newcommand{\im}{{\text{\rm im}}}

\usepackage{euscript}

\newcounter{thecounter}
\numberwithin{thecounter}{section}
\newtheorem{lemma}[thecounter]{Lemma}
\newtheorem{proposition}[thecounter]{Proposition}
\newtheorem{theorem}[thecounter]{Theorem}
\newtheorem{thm}[thecounter]{Theorem}
\newtheorem{corollary}[thecounter]{Corollary}

\newcommand{\R}{{\mathbb{R}}}

\newcommand{\C}{{\mathbb{C}}}
\newcommand{\Z}{{\mathbb{Z}}}
\newcommand{\N}{{\mathbb{N}}}

\DeclareMathOperator{\rank}{rank}
\DeclareMathOperator{\vspan}{span}
\DeclareMathOperator{\diag}{diag}

\def\a{\alpha}
\def\b{\beta}
\def\g{\gamma}
\def\d{\delta}

\newcommand{\realroots}{\Delta^\text{re}}
\begin{document}

\title{Commutator relations and structure constants for \\rank 2 Kac--Moody algebras}

\author{Lisa Carbone, Matt Kownacki, Scott H.\ Murray  and Sowmya Srinivasan}

\begin{abstract} We completely determine the structure constants between real root vectors in a rank 2 Kac--Moody algebra $\mathfrak{g}$. Our description is computationally efficient,  even in the rank 2  hyperbolic case where the coefficients of roots on the root lattice grow exponentially with height. Our approach is to extend Carter's method of finding structure constants from those on extraspecial pairs  to the rank 2 Kac--Moody case. We also determine all commutator relations involving only real root vectors in all rank 2 Kac-Moody algebras. 
The generalized Cartan matrix of $\mathfrak{g}$ is of the form
$H(a,b)= \left(\begin{smallmatrix} ~2 & -b\\
-a & ~2 \end{smallmatrix}\right)$ where $a,b\in\mathbb{Z}$ and $ab\geq 4$.  If $ab=4$, then $\mathfrak{g}$ is of affine type. If $ab>4$, then $\mathfrak{g}$ is of hyperbolic type.
Explicit knowledge of the root strings is needed, as well as a characterization of the pairs of real roots whose sums are real. We prove  that if $a$ and $b$ are both greater than one, then no sum of real roots can be a real root. We determine the root strings between real roots $\b,\g$ in $H(a,1)$, $a\geq 5$ and we determine the sets $(\mathbb{Z}_{\geq 0}\alpha+\mathbb{Z}_{\geq 0}\beta)\cap\Delta^{\re}(H(a,b))$. One of our tools is a characterization of the root subsystems generated by a subset of roots. We classify these subsystems in rank 2 Kac--Moody root systems.  We prove that every rank two infinite root system contains an infinite family of non-isomorphic symmetric rank 2 hyperbolic root subsystems $H(k,k)$ for certain $k\geq 3$, generated by either two short or two long simple roots. We also prove that a non-symmetric hyperbolic root systems $H(a,b)$ with $a\ne b$ and $ab>5$ also contains an infinite family of non-isomorphic non-symmetric rank 2 hyperbolic root subsystems $H(a\ell,b\ell)$, for certain positive integers $\ell$.

\end{abstract}
\thanks{This research made extensive use of the Magma computer algebra system.}
\thanks{The research of the first author is partly supported by  Simons Foundation Collaboration Grant, no. 422182.}

\maketitle

\section{Introduction}

An open question in the theory of Kac--Moody Lie algebras is to determine the infinite dimensional analog of a Chevalley basis. This task involves an explicit determination of commutator relations and structure constants, which is non-trivial, even in the smallest cases of rank 2. 

Furthermore, the question of determining the commutator relations between real root vectors $x_{\alpha}$ and $x_{\beta}$ corresponding to real roots $\alpha$ and $\beta$ may be reduced to the rank 2 root subsystem generated by $\alpha$ and $\beta$, where the commutator $[x_{\alpha},x_{\beta}]$ is trivial if the sum $\alpha+\beta$ is not a root and otherwise lies in the root space corresponding to $\alpha+\beta$. Thus an explicit knowledge of the commutator relations and structure constants for rank 2 Kac--Moody algebras is essential to the development of the notion of a Chevalley basis in a general Kac-Moody algebra $\mathfrak{g}$. This observation provides the setting for the current work.

Let $\Delta$ be the root system of a rank 2 Kac--Moody algebra $\mathfrak{g}$. Then $\Delta$ has generalized Cartan matrix 
$H(a,b):= \left(\begin{smallmatrix} ~2 & -b\\
-a & ~2 \end{smallmatrix}\right)$
for some $a,b\in\mathbb{Z}$. Let $S=\{\alpha_1,\alpha_2\}$ denote a basis of simple roots of $\Delta$.
If $a\neq b$, then $\Delta$ is non-symmetric and its basis consists of a long simple root and a short simple root, whereas if  $a= b$,  $\Delta$ is symmetric and its basis consists of two long simple roots. 

The generalized Cartan matrices $H(2,2)$ and $H(4,1)$ are of affine types $A_1^{(1)}$ and $A_2^{(2)}$ respectively. The remaining generalized Cartan matrices $H(a,b)$ with $ab> 4$ are all of hyperbolic type. 

The root system $\Delta$ contains two types of roots: real and imaginary. The real roots of $\Delta$ are of the form $w\alpha_i$ for some $w\in W$, where $W$ is the Weyl group of the root system. For all rank 2 Kac--Moody algebras $W\cong D_{\infty}$,  the infinite dihedral group. The additional imaginary roots will not play a significant role in this work. The real roots are supported on the branches of a hyperbola in $\R^{(1,1)}$, with a pair of branches for each root length (Figure 1).

In this paper, we completely determine the structure constants between real root vectors in rank 2 Kac--Moody algebras $\mathfrak{g}$. Our description is computationally efficient,  even in the rank 2  hyperbolic cases where the coefficients of roots grow exponentially with increasing height. 

In order to determine the Lie algebra commutators, it is  necessary to characterize the set  of real roots that are positive $\Z$--linear combinations $(\mathbb{Z}_{>0}\alpha+\mathbb{Z}_{>0}\beta)\cap\Delta^{\re}(H(a,b))$,  for  real roots $\alpha,\beta$. 
We obtain an explicit description of this set for all rank 2 Kac--Moody root systems.
A similar question was answered by Billig and Pianzola ([BP]) in less explicit form for arbitrary Kac--Moody root systems. 

Our main aim is to determine all structure constants involving only real roots. As in the finite dimensional case, we find 
systems of root vectors $\{x_{\a}\mid\a\in\Delta^\re\}$ with
$$[x_{\a},x_{\b}]= n_{\a,\b} x_{\a+\b} $$
where $n_{\a,\b}=s_{\a,\b} (p_{\a,\b}+1)$ for some $s_{\a,\b}\in\{+1,-1\}$, for all $\a,\b\in\Delta^\re$ with $\a+\b$ not an imaginary root.
The constants $p_{\a,\b}$ are determined by the root strings containing real roots, which we also determine. Our main theorem is the following (with notation as in Section~\ref{realroots}):
\begin{theorem}\label{T-sc} Let $\mathfrak{g}=\mathfrak{g}(\Delta)$ be an infinite dimensional rank 2 Kac--Moody algebra. Then the positive real roots can be enumerated as $$\{\alpha^{LL}_j ,\alpha^{LU}_j,
\alpha^{SU}_j, \alpha^{SL}_j  \mid j\in\Z,\; j\ge 0\},$$ so that:
\begin{enumerate}
\item\label{T-sc-triv} If $\mathfrak{g}$ has type $A_1^{(1)}=H(2)$, $H(m)$ for $m\geq 3$, or $H(a,b)$ for $ab>4$ with $a,b\neq 1$,  then the sum of two real roots is never a real root. Hence $n_{\a,\b}=0$ for all for $\a,\b\in\Delta^\re$ with $\a+\b\notin\Delta^\im$..
\item\label{T-sc-Ha1} Suppose $\mathfrak{g}$ has type $H(a,1)$ with $a\geq 5$. Let $s^{UL}:= s_{\a^{SU}_0,\a^{LL}_{0}}$, and let
\begin{align*}
s^U_j&:=s_{\a^{SU}_j,\a^{SU}_{j+1}}, & s^L_j:=s_{\a^{SL}_j,\a^{SL}_{j+1}}&&\text{for $j\ge0$,}
\end{align*}
be the signs of structure constants corresponding to short positive real roots whose sums are the long positive real roots $\a^{LU}_j$, 
$\a^{LL}_{j+1}$ respectively. These signs can each be chosen arbitrarily in $\{+1,-1\}$. 
This choice determines all structure constants $n_{\a,\b}$ for $\a,\b\in\Delta^\re$ with $\a+\b\in\Delta^\re$.

\item\label{T-sc-H41} Suppose $\mathfrak{g}$ has type $A_2^{(2)}=H(4,1)$. Let
\begin{align*}
s^U_{0,k}&:=s_{\a^{SU}_0,\a^{SU}_{2k+1}},& s^L_{0,k}&:=s_{\a^{SL}_0,\a^{SL}_{2k+1}} &&\text{for $k\ge0$,}\\
s^U_{1,k}&:=-s_{\a^{SU}_1,\a^{SU}_{2k}},    & s^L_{1,k}&:=-s_{\a^{SL}_1,\a^{SL}_{2k}} &&\text{for $k\ge1$,}\\
s^U_{2,k}&:=s_{\a^{SU}_2,\a^{SU}_{2k-1}},    & s^L_{2,k}&:=s_{\a^{SL}_2,\a^{SL}_{2k-1}} &&\text{for $k\ge2$.}
\end{align*}
These signs can each be chosen arbitrarily in $\{+1,-1\}$. 
This choice determines all structure constants $n_{\a,\b}$ for $\a,\b\in\Delta^\re$ with $\a+\b\in\Delta^\re$.
\end{enumerate}
\end{theorem}
See Propositions~\ref{P-allsc-triv}, \ref{P-allsc-Ha1} and~\ref{P-allsc-H41} for explicit formulas for all structure constants involving only real roots.
Case \ref{T-sc-triv} is easy to see for the symmetric cases $A_1^{(1)}$ and $H(m)$, and  was
observed by Morita (\cite{Mor}) without proof for the non-symmetric cases $H(a,b)$.
The twisted affine case \ref{T-sc-H41} is subtle, since there exist distinct pairs of real roots $(\a,\b)$ and $(\a',\b')$ with $\a+\b=\g=\a'+\b'$ and  $\g$ is a real root.  The sign of one such pair will  determine the sign of the other, but the method of proof used for finite-dimensional semisimple algebras does not generalize.

In order to make our results precise, we use two different concepts of a subsystem generated by a subset  $\Gamma$ of real roots: namely a subsystem $\Phi(\Gamma)$,  corresponding to a reflection subgroup of the Weyl group and consisting entirely of real roots; and $\Delta(\Gamma)$ consisting of all roots that can be written as an integral linear combination of elements of $\Gamma$. Such a $\Delta(\Gamma)$ subsystem contains both real and imaginary roots and corresponds to a  subalgebra of the Kac--Moody algebra.

We have classified both kinds of subsystem inside a rank 2 infinite root system, and found that the two concepts of subsystem are equivalent in almost all cases:
\begin{thm}\label{infshortrts}
Let $\Delta$ be a  rank 2 infinite root system and let $\Gamma$ be a set of real roots which generate $\Delta$, that is, $\Delta(\Gamma)=\Delta$.
Then either $\Phi(\Gamma)$ is the set of all real roots in $\Delta$ or it is the set of all \emph{short} real roots in $\Delta$.
The second case occurs only if $a=1$ or $b=1$ and $\Phi(\Gamma)$ consists of short roots. In particular, if $a=b$ then $\Delta^\re(\Gamma)=\Phi(\Gamma)$.
\end{thm}

Our classification also gives us the following result, which holds for either concept of subsystem:
\begin{thm}\label{infsubs} 
If $\Delta$ is a rank 2 hyperbolic root system, then $\Delta$ contains symmetric rank 2 hyperbolic root subsystems of type $H(k,k)$ for infinitely many distinct $k\geq 3$. 
If $\Delta$ is non-symmetric of type $H(a,b)$, then it
also contains non-symmetric rank 2 hyperbolic root subsystems of type $H(a\ell,b\ell)$ for infinitely many distinct $\ell\geq 2$.
\end{thm}
Kim and Lee also obtained an embedding theorem for rank 2 symmetric Kac--Moody algebras using different methods ([KL], Theorem 5.5).

We also classify the rank 2 $\Phi$-subsystems as finite, affine or hyperbolic
systems:
\begin{thm}\label{infbytype}
Let $\Delta$ be a rank 2 root system and let $\Gamma$ be a nonempty set of real roots in $\Delta$.
\begin{enumerate}
\item If $\Delta$ is finite, then $\Phi(\Gamma)$ is finite.
\item If $\Delta$ is affine of type $\widetilde{A}_1$, then $\Phi(\Gamma)$ has finite type $A_1$ or affine type $\widetilde{A}_1$.
\item If $\Delta$ is affine of type $\widetilde{A}_2^{(2)}$, then $\Phi(\Gamma)$ has finite type $A_1$, or affine type $\widetilde{A}_1$ or $\widetilde{A}_2^{(2)}$.
\item If $\Delta$ is hyperbolic, then $\Phi(\Gamma)$ has finite type $A_1$ or hyperbolic type.
\end{enumerate}
\end{thm}

We mention the following related works: Feingold and Nicolai (\cite{FN}, Theorem 3.1)  gave a method for generating a subalgebra corresponding to a $\Delta(\Gamma)$--type root subsystem for a certain choice of real roots in any Kac--Moody algebra.
As in Section 4 of this paper, Casselman (\cite{Cas1}) reduced the study of structure constants for Kac--Moody algebras to rank 2 subsystems. 
Tumarkin (\cite{T}) gave a classification the sublattices of  hyperbolic root lattices of the same rank. However, he requires conditions on the possible angles between roots that exclude all but a finite number of rank 2 hyperbolic root systems. In contrast, for our intended application to Kac--Moody groups,  we require the explicit construction of the embedding of the simple roots of a subsystems into the ambient system, rather than just describing its root lattice.

The authors are very grateful to Chuck Weibel for his careful reading of the MSc thesis ([Sr]) of the fourth author. This research was greatly facilitated by experiments carried out in the computational algebra systems Magma (\cite{BCFS}) and Maple (\cite{M}); the diagrams were created in Maple.

\newpage
\section{Rank 2 Kac--Moody root systems}\label{realroots}
Let $A=H(a,b)$ be the $2\times 2$ generalized Cartan matrix\footnote{This is the transpose of the generalized Cartan matrix $A$ in [ACP].}
$$
A=H(a,b) = (a_{ij})_{i,j=1,2} =
\begin{pmatrix}
2 & -b \\
-a & 2
\end{pmatrix}
$$
for positive integers $a\geq b\geq 1$, with Kac--Moody
algebra $\mathfrak{g}=\mathfrak{g}(A)$, root system
$\Delta=\Delta(A)$, and Weyl group $W=W(A)$. Let $\frak h$ denote the Cartan subalgebra of $\frak g$. Let $\langle\cdot,\cdot\rangle: \mathfrak{h}^{\ast}\longrightarrow \mathfrak{h}$ denote the natural nondegenerate bilinear pairing between  $\mathfrak{h}$ and its dual $\mathfrak{h}^*$.

Let $S=\{\alpha_1,\alpha_2\}\subset\frak h$ and $S^\vee=\{\alpha^\vee_1,\alpha^\vee_2\}\subset \frak h^*$ denote bases of simple roots and simple coroots respectively, satisfying $\langle\alpha_j,\alpha_i^{\vee}\rangle=\alpha_j(\alpha_i^{\vee})=a_{ij}$.

Then $\mathfrak{g}$ admits a symmetric invariant bilinear form $(\hspace{2pt},\hspace{2pt})$ which is unique up to normalization
(\cite{Ka}, Section II) with
$$a_{ij}= \dfrac{2(\alpha_i,\alpha_j)}{(\alpha_i,\alpha_i)}.$$
Let $D = \diag(d_1,\cdots,d_\ell)$ with $d_i = \dfrac{2 }{(\alpha_i,\alpha_i)} $ so that the matrix of this form
$$B=B(a,b)= DA=(d_i a_{ij})_{i,j=1,2} =\left( \dfrac{2a_{ij} }{(\alpha_i,\alpha_i)} \right)_{i,j=1,2}=\left(\begin{array}{rr}  2a/b&-a\\-a&2\end{array}\right)$$
is a symmetrization of $A$.  Note that we have normalized so that $\a_2$ is the short simple root  with $(\a_2,\a_2)=2$ and $\a_1$ is the long simple root with $(\a_1,\a_1)=2a/b$.

We identify $\a_i^\vee$ with $\dfrac{2\a_i}{(\a_i,\a_i)}$. For any real root $\a$,  we identify $\a^\vee$ with $\dfrac{2\a}{(\a,\a)}$.

We have the simple root reflections
$$w_j(\alpha_i)=\alpha_i-a_{ij}\alpha_j$$
for $i=1,2$ 
with matrices with respect to $S$
$$[w_1]_S = \left(\begin{array}{rr} -1&b\\0&1\end{array}\right),\qquad
[w_2]_S = \left(\begin{array}{rr}  1&0\\a&-1\end{array}\right).$$
The Weyl group $W=W(A)$ is the group generated by the simple root reflections $w_1$ and $w_2$.

When $ab<4$, $A$ is positive definite and so $\Delta$ is finite.
When $ab=4$, $A$ is
positive semi-definite but not positive definite and so $\Delta$ is affine. When $ab>4$, $A$ is indefinite but every proper generalized Cartan submatrix is
positive definite, and so $A$ is  hyperbolic. Without loss of generality, we assume that $a\ge b$.

The set of real roots is $$\Delta^\re=W\a_1\cup W\a_2.$$
The set of imaginary roots is 
$$\Delta^\text{im}=\{\alpha \in \mathbb{Z}\a_1+\mathbb{Z}\a_2 \mid \a\ne0\text{ and }|\a|^2 \le 0\}.$$

A diagram of the hyperbolic root system $H(5,1)$ is  given in Figure 1.

\begin{figure}[h]
\begin{center}
{\includegraphics[width=2.9in]{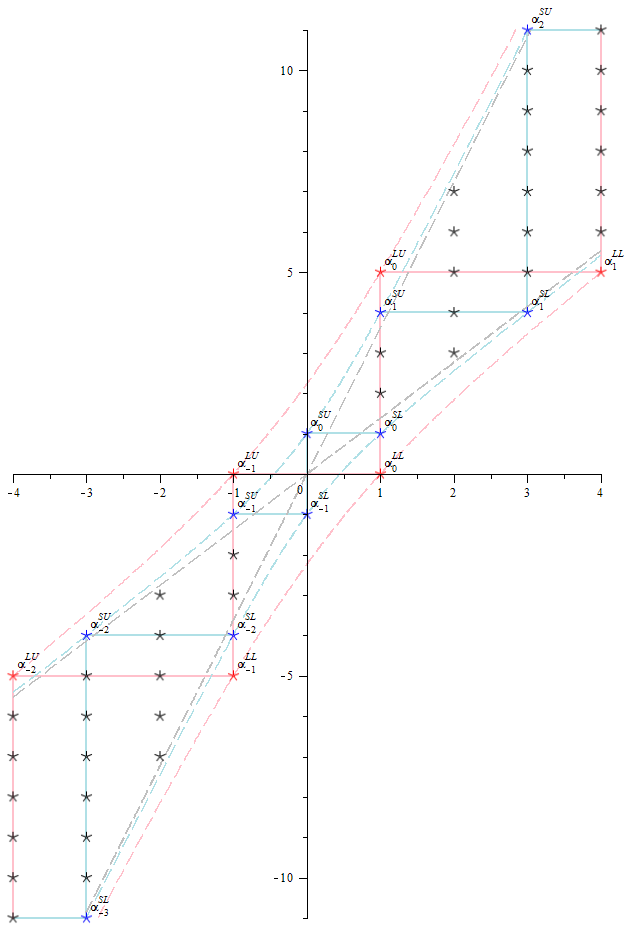}}
\caption{Root system of type $H(5,1)$}
\label{H51}
\end{center}
\end{figure}

Every root $\alpha\in\Delta$ has an expression of the form $\alpha= k_1\alpha_1+k_2\alpha_2$ where the $k_i$ are either all $\geq 0$, in which case $\alpha$ is called {\it positive}, or all $\leq 0$, in which case $\alpha$ is called {\it negative}. The positive roots are denoted $\Delta_+$, the negative roots $\Delta_-$.

Now $|\a_1|^2=2a/b$ and $|\a_2|^2=2$. 
So all real roots $x\alpha_1+y\alpha_2$ in the orbit $W\a_1$ satisfy
$$ {a}x^2 -abxy+by^2 = b,$$
and all real roots $x\alpha_1+y\alpha_2$ in the orbit $W\a_2$ satisfy
$$ {a}x^2 -abxy+by^2 = a.$$
These curves are displayed in Figures~\ref{H51}--\ref{Hyp} as blue (resp.\ red) dotted lines. These curves are elliptical for finite systems, straight lines for affine systems, and hyperbolas for hyperbolic systems.
If $\Delta$ is nonsymmetric ($a>b$) the roots in $W\a_1$ are called \emph{long} and the roots in $W\a_2$ are called \emph{short}.
If $\Delta$ is symmetric ($a=b$) then all roots are considered to be \emph{long}.
Note that (with the exception of $A_2$), the real roots fall into two distinct orbits under the action of $W$. The figures use red for the orbit of $\alpha_1$, blue for the orbit of $\alpha_2$, and black for the imaginary roots. The horizontal lines indicate the action of $w_1$ while the vertical lines indicate the action of $w_2$.

For $j\in\mathbb{Z}$, we define
\begin{align*}
\alpha^{LL}_j &:= (w_1w_2)^j\alpha_1,&\alpha^{LU}_j &:= (w_2w_1)^jw_2\alpha_1\\
\alpha^{SU}_j &:= (w_2w_1)^j\alpha_2, &\alpha^{SL}_j &:= (w_1w_2)^jw_1\alpha_2.
\end{align*}
All real roots are given by these four sequences. If $ab\ge4$, then these are all distinct, and  a root is positive  if and only if $j \geq 0$. 

The following lemma characterizes the real roots in terms of recursive sequences $\eta_j$ and $\gamma_j$. Values of these sequences for small $j$ are given in Table~\ref{smallm}.
\begin{table}[h]
{\tiny
$$\begin{array}{r|rr}
j & \gamma_j & \eta_j \\\hline
0 & 0 & 1 \\
1 & 1 & ab - 1 \\
2 & ab - 2 & a^2b^2 - 3ab + 1 \\
3 & a^2b^2 - 4ab + 3 & a^3b^3 - 5a^2b^2 + 6ab - 1 \\
4 & a^3b^3 - 6a^2b^2 + 10ab - 4 & a^4b^4 - 7a^3b^3 + 15a^2b^2 - 10ab + 1 \\
5 & a^4b^4 - 8a^3b^3 + 21a^2b^2 - 20ab + 5 & a^5b^5 - 9a^4b^4 +
28a^3b^3 - 35a^2b^2 + 15ab - 1 \\
6 & a^5b^5 - 10a^4b^4 + 36a^3b^3 - 56a^2b^2 + 35ab - 6 & a^6b^6 -
11a^5b^5 + 45a^4b^4 - 84a^3b^3 + 70a^2b^2 - 21ab + 1 \\
7 & a^6b^6 - 12a^5b^5 + 55a^4b^4 - 120a^3b^3 + 126a^2b^2 - 56ab + 7 &
a^7b^7 - 13a^6b^6 + 66a^5b^5 - 165a^4b^4 + 210a^3b^3 - 126a^2b^2 +
28ab - 1 \\
8 & a^7b^7 - 14a^6b^6 + 78a^5b^5 - 220a^4b^4 + 330a^3b^3 - 252a^2b^2
+ 84ab - 8 & a^8b^8 - 15a^7b^7 + 91a^6b^6 - 286a^5b^5 + 495a^4b^4 -
462a^3b^3 + 210a^2b^2 - 36ab + 1 
\end{array}$$}
\caption{Values of $\eta_j$ and $\gamma_j$ for small $j$}
\label{smallm}
\end{table}
\begin{lemma}\label{acp} ([ACP], Lemmas 3.2 and 3.3)$\;$ For all integers $j$,
\begin{align*}
\alpha_j^{LL}&= \eta_{j}\alpha_1+a\gamma_{j}\alpha_2,& 
\alpha_j^{LU}&= \eta_{j}\alpha_1+a\gamma_{j+1}\alpha_2,\\
\alpha_j^{SU}&= b\gamma_j\alpha_1+\eta_{j}\alpha_2,& 
\alpha_j^{SL}&= b\gamma_{j+1}\alpha_1+\eta_{j}\alpha_2,
\end{align*}
where
\begin{enumerate}
\item $\gamma_0=0$, $\gamma_1=1$, $\eta_0=1$, $\eta_1=ab-1$;
\item $\eta_{j}=ab\gamma_{j}-\eta_{j-1}$;
\item $\gamma_{j}=\eta_{j-1}-\gamma_{j-1}$;
\item both sequences $X_j=\eta_j$ and $\gamma_j$ satisfy the recurrence relation
$$X_j=(ab-2)X_{j-1}-X_{j-2}.$$
\end{enumerate}
\end{lemma}
Note that these are both generalized Fibonacci sequences provided that $ab>4$. In particular, $\gamma_j$ is the Lucas sequence with parameters $P=ab-2,Q=1$.

The following  useful  lemma gives negatives of roots:
\begin{lemma}\label{L-neg} For all $j\in\mathbb{Z}$, $\gamma_{-j}=-\gamma_{j}$ and $\eta_{-j}=-\eta_{j-1}$. Also
$$
-\alpha^{LL}_j = \alpha^{LU}_{-j-1}, \quad -\alpha^{LU}_j = \alpha^{LL}_{-j-1},\quad 
-\alpha^{SU}_j = \alpha^{SL}_{-j-1}, \quad -\alpha^{SL}_j = \alpha^{SU}_{-j-1}.
$$
\end{lemma}

\section{Sums of real roots}\label{rootsums}
Let $\Delta$ be an infinite rank 2 root system of type $H(a,b)$ with $a\ge b$ and $ab\ge4$.
In this section we determine all real roots $\alpha,\beta\in\Delta$ for which  $\alpha+\beta$ is also a real root.

We will split our analysis into two cases: that in which $a\geq b >1$, and that in which $a>b=1$. We find in the first case that the sum of two real roots is never a real root, and in the second case that there are certain $\beta\in\realroots$ so that $\beta\pm \alpha_i \in \realroots$.

\subsection{The case $a\ge b>1$}

\begin{lemma}\label{staircase} If $a\ge b>1$, then
\begin{eqnarray*}
0=b\gamma_0<\eta_0<b\gamma_1<\eta_1<b\gamma_2<\cdots,\\
0=a\gamma_0<\eta_0<a\gamma_1<\eta_1<a\gamma_2<\cdots.
\end{eqnarray*}
In fact the gaps between sequence elements are nondecreasing, that is, for $j\ge0$,
\begin{align*}
 \eta_{j+1}-b\gamma_{j+1}&\ge b\gamma_{j+1}-\eta_j\ge\eta_j-b\gamma_j,\\
 \eta_{j+1}-a\gamma_{j+1}&\ge a\gamma_{j+1}-\eta_j\ge\eta_j-a\gamma_{j}.
\end{align*}
\end{lemma}
\begin{proof}
 To see that the gaps in the sequences are nondecreasing, we apply Lemma~\ref{acp} as follows: 
\begin{align*}
\eta_{j+1}-b\gamma_{j+1}&= (a-1)b\gamma_{j+1}-\eta_j
\ge b\gamma_{j+1}-\eta_m= (b-1)\eta_j-b\gamma_{j}\ge \eta_j-b\gamma_{j}. 
\end{align*} 
The other result is similar.
\end{proof}

The inequalities in Lemma~\ref{staircase} show that the real roots have the "staircase pattern'' shown in Figure~\ref{bgt1}.
\begin{figure}[h]
\begin{center}
{\includegraphics[width=3.4in]{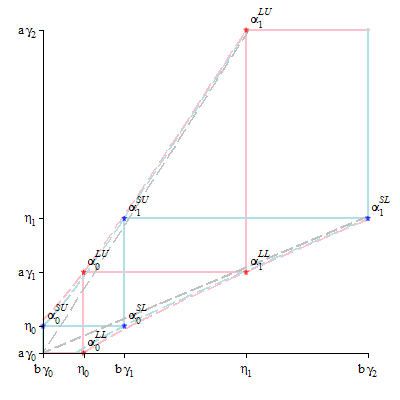}}
\caption{The positive real roots for $H(a,b)$ with $a\ge b>1$}
\label{bgt1}
\end{center}\end{figure}

\begin{proposition}\label{sumbne1}
If $a\ge b>1$ and $\a,\beta\in\Delta^\re$, then $\a+\b\notin\Delta^\re$.
\end{proposition}
\begin{proof}
We can write $\a=w\a_i$ for $i=1$ or 2, and some $w\in W$. 
We may also replace $\beta$ by $w^{-1}\beta$. Thus we wish to  determine the $\beta\in\Delta^{\re}$ for which $\alpha_i+\beta\in\Delta^\re$.
Replacing $\beta$ by $-\beta$ if $\beta\in\Delta^\re_-$, we wish to determine the $\beta\in\Delta_+^{\re}$ for which $\alpha_i\pm\beta\in\Delta^\re$.
Thus we may take $\a=\pm\a_i$ and $\b\in\Delta^\re_+$.
From Figure~\ref{bgt1}, it is clear that $\b\pm\a_i\in\Delta^\re$ only when  one of the differences $\eta_{j+1}-\eta_j$ or
$\gamma_{j+1}-\gamma_j$ equals 1. By Lemma~\ref{staircase}, this cannot occur.
\end{proof}

\subsection{The case $a>b=1$}
This case  is considerably more intricate. 
We define real functions
$\Psi_\pm(x) := \frac12\left((x-2)\pm\sqrt{x(x-4)}\right)$,
for which $\psi_\pm=\Psi_\pm(ab)$ are the characteristic roots of the recurrence equation in
Lemma~\ref{acp}(iv).

The following lemma gives a bound on these parameters. 
\begin{lemma}\label{bound} If $ab>4$, then $\psi_+>2.61$ and $0<\psi_-< 0.45$. . 
\end{lemma}
\begin{proof}
We can use the derivative of $\Psi_\pm(x)$ to show that
$\Psi_+$ is increasing for $x\ge5$ and $\Psi_-$ is positive and decreasing for $x\ge5$.
Hence $\psi_+\ge\Psi_+(5)>2.61$ and $0<\psi_-\le\Psi_-(5)<0.45.$
\end{proof}
Define
$$ \lambda := \frac{\psi_+}{\psi_+-1}, \qquad \mu:=\frac{1}{\sqrt{ab(ab-4)}}.$$
The following lemma is an easy consequence of Lemma~\ref{bound} and shows that the sequences $\eta_j$ and $\gamma_j$ are each within a small constant of being exponential with basis $\psi_+$.
\begin{lemma} If $ab>4$, then, for $j\ge0$,
\begin{align*}
\lambda\psi_+^j-1.62<\eta_j&<\lambda\psi_+^j,\\
\mu\psi_+^j-0.45<\gamma_j&<\mu\psi_+^j,
\end{align*}
where $\psi_+>2.61$, $1 < \lambda < 1.62$, and $0 < \mu < 0.45$.
\end{lemma}

The following lemma now follows,  showing that the roots have the "staircase pattern'' shown in Figure~\ref{beq1}.
\begin{lemma}\label{staircase2} (\cite{CKMS},~\cite{Sr}) If $a>4$ and $b=1$, then
{\re}
\begin{align*}
0&=\gamma_0<\eta_0=\gamma_1<\gamma_2<\eta_1<\gamma_3<\eta_2<\cdots,\\
0&=a\gamma_0<\eta_0<\eta_1<a\gamma_1<\eta_2<a\gamma_2<\eta_3<a\gamma_3<\cdots.
\end{align*}
\end{lemma}

\begin{figure}
\begin{center}
{\includegraphics[width=3.4in]{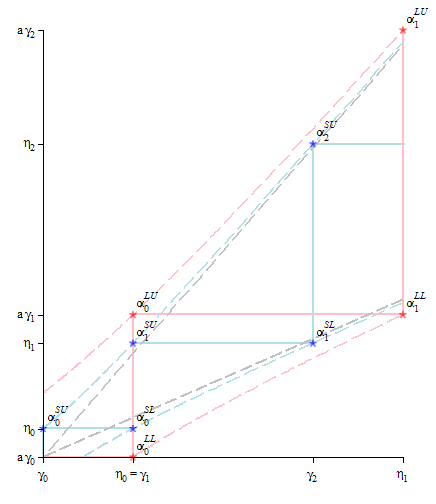}}
\caption{The positive real roots for $H(a,1)$ with $a>4$}
\label{beq1}
\end{center}\end{figure}

 We now use the above results to determine the $\beta\in\realroots_+$ for which $\beta\pm \alpha_i\in\realroots$, for $i=1,2$.

\begin{theorem}
\label{T-beq1}
If $a>4$, $b=1$ and $\beta\in\Delta_+^{\re}$ then
\begin{enumerate}
\item $\beta+\alpha_1\in\Delta^{\re}$ if and only if  $\beta=\alpha_2$;
\item $\beta-\alpha_1\in\Delta^{\re}$ if and only if  $\beta=\alpha_1+\alpha_2$;
\item $\beta+\alpha_2\in\Delta^{\re}$ if and only if  $\beta=\alpha_1$ or $\alpha_1+(a-1)\alpha_2$;
\item $\beta-\alpha_2\in\Delta^{\re}$ if and only if  $\beta=\alpha_1+\alpha_2$ or $\alpha_1+a\alpha_2$.
\end{enumerate}
\end{theorem}

\begin{proof} For (i),  note that if $\beta = \alpha_2$, then $\beta+\alpha_1 = \alpha_1+\alpha_2 = \alpha_0^{SL}$, by Table 1.  Thus $\beta+\alpha_1$ is  a real root. 

Conversely, let $\beta\in\realroots$ such that $\beta+\alpha_1 \in\realroots$. 
First suppose that $\beta=\alpha_j^{SU}$ for some $j$. By Lemma~\ref{acp}, we have 
$\beta_j = \gamma_j\alpha_1 + \eta_j\alpha_2.$ Then $$\beta+\alpha_1 = (\gamma_j+1)\alpha_1+\eta_j\alpha_2.$$
If $\beta+\alpha_1$ is a long root on a lower branch, then again by Lemma~\ref{acp} we have 
$\beta+\alpha_1 = \eta_k \alpha_1 + a\gamma_k\alpha_2$ for some $k$. 
Then we must have $\gamma_{j}+1 = \eta_k$ and $\eta_j = a\gamma_k$, but by Lemma~\ref{staircase}  there are no $j,k$ such that $\eta_j = a\gamma_k$. Similarly, there are no $j,k$ such that $\gamma_j+1 = \eta_k$ and 
$\eta_j = a\gamma_{k+1}$, so $\beta+\alpha_1$ cannot be a long root on an upper branch. 
If $\beta+\alpha_1$ is a short root on an upper branch, then Lemma~\ref{acp} implies that $\gamma_j+1 = \gamma_k $ and $\eta_j = \eta_k$, but again, 
no such $j,k$ can exist: if $\beta+\alpha_1$ 
is a short root on an upper branch, then we must have $j,k$ so that
$$ \gamma_j+1 = \gamma_{k+1} \,\,\, \text{and} \,\,\, \eta_j = \eta_k.$$
The second of these conditions implies that $j=k$. Then by the first of these conditions we have $\gamma_{j+1}-\gamma_j = 1$. Then $j$ must be $0$. So 
$$\beta+\alpha_1 = \alpha_0^{SL} = \alpha_1 + \alpha_2,$$
so $\beta$ is $\alpha_2$. Similarly we may check that if $\beta = \alpha_{j}^{LL}$, $\beta=\alpha_j^{LU}$, or $\beta=\alpha_j^{SL}$ for some $j>0$, then $\beta+\alpha_1 \not\in\realroots$. 

Following similar reasoning, we can check that the second, third and fourth claims hold. \end{proof}

By a straightforward case-by-case argument, we can now prove the following result about lengths of sums of roots. This is easily proved for $H(4,1)$, since in this case every short root $x\a_1+y\a_2$ is on the lines
$y=4x\pm1$, and every long root is on the lines $y=4x\pm2$.

\begin{thm} \label{sums}
Let $\Delta$ be an infinite rank 2 root system.
\begin{enumerate}
\item If $\alpha,\beta,\alpha+\beta\in\Delta^{\re}$ with $\alpha$ and $\beta$ short, then $\alpha+\beta$ is long.
\item If $\alpha,\beta,\alpha+\beta\in\Delta^{\re}$ with $\alpha$ short and $\beta$ long, then $\alpha+\beta$ is short.
\item If $\alpha,\beta\in\Delta^{\re}$ with $\alpha$ and $\beta$ long, then $\alpha+\beta\notin\Delta^{\re}$.
\end{enumerate}
\end{thm}

We note that (i) and (iii) are not true in finite root systems of type $A_2$ or $G_2$. However there is a slightly weaker result that holds in any symmetrizable system:
\begin{thm}
Let $\Delta$ be a symmetrizable root system and suppose $\alpha,\beta,\alpha+\beta\in\Delta^{\re}$.
\begin{enumerate}
\item If $|\alpha|^2=|\beta|^2$, then $|\alpha+\beta|^2=a|\alpha|^2$ for some positive integer $a$.
\item If $|\alpha|^2\ne|\beta|^2$, then $|\alpha+\beta|^2=\min(|\alpha|^2,|\beta|^2)$.
\end{enumerate}
\end{thm}
\begin{proof}
We only need to consider the rank 2 subsystem $\mathbb{Z}\{\a,\b\}\cap\Delta$. These results are easily shown to be true if the subsystem has finite type $A_2$, $B_2$, or $G_2$, and they follow from the previous theorem if the subsystem is infinite.
\end{proof}

We mention the following lemma which has many  applications.

\begin{lemma} Let $\frak g$ be a symmetrizable Kac--Moody algebra. Let 
$\a,\b\in\Delta^{\re}$ be real roots with $\a\neq \b$. If $\a+\b$ and $\a-\b$ are not roots, then no other integral linear combination of $\a$ and $\b$ is a root.
 \end{lemma}
 
\begin{proof} The $\a$-string through $\b$ consists only of $\b$ and the $\b$-string through $\a$ consists only of $\a$. Thus no other integral linear combination of $\a$ and $\b$ can be a root. \end{proof}

\section{Subsystems}\label{Rk2subsys}

Root systems can be used to construct Coxeter groups, Kac--Moody algebras and Kac--Moody groups. 
These structures lead to two different concepts of 
subsystem. In this section, we describe two distinct types of root subsystem and show  that the two concepts usually coincide, but not always. We also classify all subsystems of infinite rank 2 root systems.

 Suppose that $\Delta$ is a symmetrizable root system with simple roots $\Pi=\{\alpha_1,\dots ,\alpha_{\ell}\}$.
Let $W=W(\Delta)$ be the Weyl group and  let $\Delta^\re=W\Pi$ denote the real roots.

For  $\Gamma\subseteq\Delta^\re$,  the  {\it reflection subgroup} generated by $\Gamma$ is defined as
  $$W_\Gamma = \langle w_\alpha : \alpha \in \Gamma \rangle.$$
Then $W_\Gamma$ is also a Coxeter group
and its root system is
 $$\Phi(\Gamma)= W_\Gamma\cdot \Gamma.$$ 
We call $\Phi(\Gamma)$ a \emph{$\Phi$-subsystem} (also noted in [C], Proposition 7). Note that a $\Phi$-subsystem consists entirely of real roots.

Let $\alpha$ be any real root. Then there is a corresponding pair of root vectors $x_\alpha$ and $x_{-\alpha}$ in $\mathfrak{g}=\mathfrak{g}(\Delta)$ which generate a subalgebra
isomorphic to $\mathfrak{sl}_2$.
We denote this subalgebra by $\mathfrak{sl}_2(\alpha)$.
Now let $\Gamma\subseteq \Delta^{\re}$. We may define the fundamental Kac--Moody subalgebra $\mathfrak{g}_\Gamma$ corresponding to $\Gamma$ to be the subalgebra generated by 
$$\{ \mathfrak{h}, \mathfrak{sl}_2(\alpha) : \alpha \in \Gamma \}$$

 Then $\mathfrak{g}_\Gamma$ is a Kac--Moody algebra and its root system is
$$\Delta(\Gamma)= \mathbb{Z}\Gamma \cap \Delta,$$
that is the set of all roots in $\Delta$ that can be written as an integer linear combination of elements of $\Gamma$.
We call this a \emph{$\Delta$-subsystem}. The Kac--Moody subalgebra of \cite{FN}, Theorem 3.1 is of this type. We also define $\Delta^\re(\Gamma)=\mathbb{Z}\Gamma\cap\Delta^\re$.

\subsection{Subsystems corresponding to submatrices}
In this section we discuss one of the easiest ways to construct subsystems, and show that the notions of $\Delta$- and $\Phi$-subsystems coincide in this case.
Let $\mathfrak{g}$ be a Kac--Moody algebra with generalized Cartan matrix $A=(a_{ij})_{i,j\in I}$, $I=\{1,2,\dots ,\ell\}$, Cartan subalgebra $\mathfrak{h}$ of dimension $2\ell-\rank(A)$, simple roots $\Pi=\{\alpha_1,\dots,\alpha_{\ell}\}\subseteq \mathfrak{h}^{\ast}$ and simple 
coroots $\Pi^{\vee}=\{\alpha_1^{\vee},\dots,\alpha_{\ell}^{\vee}\} \subseteq\mathfrak{h}$. Let $Q$ denote the root lattice of $\mathfrak{g}$ and let $\mathfrak{g}=\mathfrak{h}\oplus\left(\bigoplus_{\alpha\in Q\backslash\{0\}}\mathfrak{g}^{\alpha}\right)$ denote the root space decomposition. Let $\Delta$ denote the set of all roots.

\medskip 
\noindent  Let $B=(a_{ij})_{i,j\in K}$ be a submatrix of $A$ for some $K\subseteq I$ with $|K|=\ell_0$. Let $\mathfrak{h}(B)$ be a subspace of $\mathfrak{h}$ of dimension $2\ell_0-\rank(A_0)$ containing $\Pi(B)^{\vee}=\{\alpha_i^{\vee} \mid i\in K\},$ and such that $\Pi(B)=\{\alpha_i|_{\mathfrak{h}(B)^{\ast}} \mid i \in K\}$ is linearly independent. 
Set $Q_0=\bigoplus_{i\in K} \mathbb{Z}\alpha_i$. Then 
$$\mathfrak{g}_0\cong \mathfrak{h}(B)\oplus\left(\bigoplus_{\alpha\in Q_0\backslash\{0\}} \mathfrak{g}^{\alpha}\right)$$
is the Kac--Moody algebra of $B$ with Cartan subalgebra $\mathfrak{h}(B)$, simple roots $\Pi(B)$ and simple coroots $\Pi(B)^{\vee}$ ([K], Exercise 1.2).
We identify $Q_0\subseteq Q\subset \mathfrak{h}^\ast$ with $\Z\Pi(B)\subset\mathfrak{h}(B)^\ast$ in the obvious way.







\begin{proposition} (Proposition 6, [Mo]) $\;$Let $A=(A_{ij})_{i,j\in I}$ be a generalized Cartan matrix with $I=\{1,2,\dots ,\ell\}$. Let 
$K\subset I$, $K\neq \varnothing$. Let $\Pi(A)=\{\alpha_1,\dots ,\alpha_{\ell}\}$ be the simple roots of the root system $\Delta(A)$. Let $B=(A_{ij})_{i,j\in K}$. Let $\Delta(B)$ denote the root system corresponding to $B$. Let 
$\Pi(B)=\Pi(A)\cap\Delta(B)$. Then $\Pi(B)\neq\varnothing$ and $\Delta(B)$  has the properties:
$$\Delta(B)=\Delta\cap \Z\Pi(B)$$
where $\Z\Pi(B)$ denotes all integral linear combinations of $\Pi(B)$
and
\begin{align*}
\Delta(B)^{\re}&=\Delta^{\re}\cap \Z\Pi(B),\\
\Delta(B)^{\im}&=\Delta^{\im}\cap \Z\Pi(B).
\end{align*}
\end{proposition}

\begin{proposition}  Using the notation above, define $\Phi(B)=W_{\Pi(B)}(\Pi(B))$. Then 
$$\Delta(B)^{\re}=\Phi(B).$$
\end{proposition}

\begin{proof}  Our claim is that
$$\Delta^{\re}\cap \Z\Pi(B) = W_{\Pi(B)}(\Pi(B)).$$
The inclusion $\Delta(B)^{\re}\subseteq \Phi(B)$ is clear. To prove the reverse inclusion, let $\alpha\in \Phi(B)$. Then $\alpha\in \Delta^{\re}$ and $\alpha\in \Z\Pi(B)$ by definition. Hence $\Phi(B)\subseteq \Delta(B)^{\re}$.
\end{proof}

The following lemma establishes a useful property of $\Delta(B)$ subsystems.

\begin{lemma} $\Delta(B)^{\re}$ consists of the roots that are  integral sums of the simple roots corresponding to $B$. 
\end{lemma}
\begin{proof} We have $\Delta(B)^{\re}=\Delta^{\re}\cap \Z\Pi(B)$, but since $B$ is a subsystem arising from a submatrix of the generalized Cartan matrix, $B$ has an associated root lattice $Q_0=\bigoplus_{i\in K} \mathbb{Z}\alpha_i$ which is closed under taking integral sums. (See also Section 4, in particular Lemma 6, of [C]).
\end{proof}

Since $\Delta(B)^{\re}=\Phi(B)$, we have the following.

\begin{proposition} Using the notation above,  $\Phi(B)$ consists of the roots that are integral sums of the simple roots corresponding to $B$.

\end{proposition}

\begin{proof} Let $\alpha,\beta\in \Phi(B)$.
We recall that $$-p\a+\beta,\ \dots,\ \b-\a,\ \beta, \ \a+\b, \ \dots ,\ q\a+\beta$$ is the $\alpha$--string through $\beta$. We claim that for $s,t\in\Z$, $s\alpha+t\beta\in \Z\Pi(B)$. Let
$$\alpha=\sum_{i\in K} a_i\alpha_i \quad\text{and  \ }
\beta=\sum_{i\in K} b_i\alpha_i.$$
Writing elements of the root string in terms of their coordinates on the root lattice $Q_0$, it is clear that they are all elements of $\Z\Pi(B)$. For example
$-p\a+\beta=\sum_{i\in K}(-pa_i+b_i)\alpha_i.$
\end{proof}

\subsection{Classification of $\Phi$-subsystems in rank 2}
Our next step is to classify the $\Phi$-subsystems in any infinite rank 2 root system using explicit formulas for the Weyl group reflections. Let $\Delta$ be a root system of type $H(a,b)$ for $a\ge b$ and $ab\ge4$.
Let $\Gamma\subseteq \Delta^{\re}$ be nonempty.

First we note that $\Phi(\Gamma)$ is closed under negation, since $w_\a\a=-\a$. So, using the formulas of Lemma~\ref{L-neg},
$$\Phi(\Gamma)=\{\alpha^{LL}_j,\alpha^{LU}_{-j-1},\alpha^{SU}_k,\alpha^{SL}_{-k-1} \mid j\in I^L,\, k\in I^S\},$$
for some index sets $I^L,I^S\subseteq \mathbb{Z}$.
Every  real root has the form $\a=w\a_i$ for $i=1,2$ and $w\in W$, so the reflection in $\alpha$ is $w_\a=ww_iw^{-1}$.
We obtain the following formulas for the reflections corresponding to each real root:
$$w^{LL}_j = w^{LU}_{-j-1}= (w_1w_2)^{2j}w_1, \qquad  w^{SU}_j=w^{SL}_{-j-1}=(w_2w_1)^{2j}w_2.$$
We can use this to easily prove formulas for the action of a reflection on a real root:
\begin{lemma}\label{reflact} For all $j,k\in\mathbb{Z}$,
\begin{align*}
w^{LL}_k\a^{LL}_j&=-\a^{LL}_{2k-j},& w^{SU}_k\a^{SU}_j&= -\a^{SU}_{2k-j},\\
w^{LL}_k\a^{SU}_j&=-\a^{SU}_{-2k-j-1},& w^{SU}_k\a^{LL}_j&= -\a^{LL}_{-2k-j-1}.
\end{align*}
\end{lemma}

\begin{lemma} \label{reflincl} Given integers $j$ and $k$:
\begin{enumerate} 
\item\label{reflincl-l} If $j,k\in I^L$, then $j+(k-i)\mathbb{Z}\subseteq I^L$.
\item\label{reflincl-s} If $j,k\in I^S$, then $j+(k-j)\mathbb{Z}\subseteq I^S$.
\item\label{reflincl-ls}If $j\in I^L$, $k\in I^S$, then $j+(2j+2k+1)\mathbb{Z}\subseteq I^L$ and $k+(2j+2k+1)\mathbb{Z}\subseteq I^S$.
\end{enumerate}
\end{lemma}
\begin{proof}
Suppose $I^S$ contains $\ell:=j+(n-1)(k-j)$ and $m:=j+n(k-j)$. Then
Lemma~\ref{reflact} shows that $j+(n+1)(k-j)=2\ell-m\in I^S$ and
$j+(n-2)(k-j)=2m-\ell \in I^S$. Part (i) now follows by induction, using separate cases for $\mathbb{Z}_{\geq 0}$ and $\mathbb{Z}_{< 0}$. Part (ii) is similar.

Let $d:=2j+2k+1$. Now suppose $j,j+nd\in I^L$ and $k, k+nd\in I^S$.
Then $j-(n+1)d= -2k-(j+nd)-1\in I^L$ and so $j+(n+1)d =2j-(j-(n+1)d)\in I^S$.
Similar arguments show that $j+(n-1)d\in I^S$ and $k+(n\pm1)d\in I^L$.
Part (iii) now follows by induction, using separate cases for $\mathbb{Z}_{\geq 0}$ and $\mathbb{Z}_{< 0}$.
\end{proof}

The following result shows how to classify the $\Phi$-subsystems in terms of their index sets. The proof is routine:
\begin{proposition}\label{Iclass} (\cite{CKMS},\cite{Sr})
\begin{enumerate} 
\item If $I^S$ is empty, then $I^L=r+d\mathbb{Z}$ for some $r,d\in\mathbb{Z}$ with $d\ge0$ and $0\le r<d$.
\item If $I^L$ is empty, then $I^S=r+d\mathbb{Z}$ for some $r,d\in\mathbb{Z}$ with $d\ge0$ and $0\le r<d$.
\item Otherwise, $I^L=r+(2d+1)\mathbb{Z}$ and $I^S=d-r+(2d+1)\mathbb{Z}$ for some $d\ge0$ and $-d\le r\le d$.
\end{enumerate}
\end{proposition}

\begin{thm}\label{subsys}
Let $\Delta$ be an infinite rank 2 root system of type $H(a,b)$ with $a\ge b$ and $ab\ge4$.
Every nonempty $\Phi$-subsystem of $\Delta$ has simple roots, Cartan matrix, and inner product matrix given by one of the rows in Table~\ref{T-Phisubs}
where $\delta_d:= \eta_{d}-\eta_{d-1}$ and  $\epsilon_d:= \gamma_{d+1}-\gamma_d$.
In particular  all $\Phi$-subsystems of $\Delta$ have rank at most $2$.
\end{thm}
\begin{table}
\begin{tabular}{l||l|l|l|l}
Type & Integer conditions & Simple roots & Cartan Matrix & Inner product matrix\\\hline
$\text{I}_L$&$r$ arbitrary & $\a^{LL}_r$ & $A_1$ & $\frac{a}{b}A_1$\\
$\text{I}_S$&$r$ arbitrary & $\a^{SU}_r$ & $A_1$ & $A_1$\\
$\text{II}_L$ & $d>0$, $0\le r<d$ & $\a^{LL}_r,\a^{LU}_{d-r-1}$& $H(\delta_d,\delta_d)$& $\frac{a}{b} H(\delta_d,\delta_d)$\\
$\text{II}_S$&$d>0$, $0\le r<d$ & $\a^{SU}_r,\a^{SL}_{d-r-1}$& $H(\delta_d,\delta_d)$& $H(\delta_d,\delta_d)$\\
$\text{II}_{LS}$&$d\ge 0$, $-d\le r\le d$ & $\a^{LL}_r,\a^{SU}_{d-r}$& $H(a\epsilon_d, b\epsilon_d)$ & $B(a\epsilon_d, b\epsilon_d)$
\end{tabular}
\vspace{2mm}
\caption{$\Phi$-subsystems of rank 2 root systems}\label{T-Phisubs}
\end{table}

\begin{proof}
Let $\Phi'$ be a $\Phi$-subsystem of $\Delta$.
First suppose that $\Phi'\subseteq W\a_1$. 
Then Proposition~\ref{Iclass}(i) implies that 
$\Phi'=\{\alpha^{LL}_j,\alpha^{LU}_{-j-1} \mid j\in r+d\mathbb{Z}\},$
for some $d\ge0$ and $0\le r<d$.
If $d=0$, this gives us type $\text{I}_L$.
Otherwise it is easily shown that every positive root in $\Phi'$ is a
positive linear combination of $\a^{LL}_r$ and $\a^{LU}_{d-r-1}$, so these roots forms a basis. The Cartan matrix and inner product matrix can be computed directly from the basis. 

For example, if the Cartan matrix is $(c_{ij})$ then
\begin{align*}
c_{12} &= \frac{2(\a^{LL}_r,\a^{LU}_{d-r-1})}{(\alpha_r^{LL},\alpha^{LL}_r)}
  = \frac{b}{a} (\a^{LL}_r,\a^{LU}_{d-r-1})
  =\frac{b}{a} \left((w_1w_2)^r\a_1,(w_1w_2)^r\a^{LU}_{d-1}\right)
  =\frac{b}{a} (\a_1,\a^{LU}_{d-1})\\
  &= \frac{b}{a}\left(\begin{matrix}1&0\end{matrix}\right) \left(\begin{matrix}2a/b&-a\\-a&2\end{matrix}\right) 
  \left(\begin{matrix}\eta_{d-1}\\a\gamma _{d}\end{matrix}\right)
   = \frac{b}{a}\left(2\frac{a}{b}\eta_{d-1}-a^2\gamma_d\right)
 = 2\eta_{d-1}-ab\gamma_d
  = \eta_{d-1}-\eta_d =-\delta_d,
\end{align*} 
where the second last equality follows from Lemma~\ref{acp}(ii).
This gives type $\text{II}_L$.

Similarly we get types $\text{I}_S$ and $\text{II}_S$ from  Proposition~\ref{Iclass}(ii), and type $\text{II}_{LS}$ from Proposition~\ref{Iclass}(iii).
\end{proof}

As a corollary, we obtain the following.

\begin{thm}\label{infsubs} 
If $\Delta$ is a rank 2 hyperbolic root system, then $\Delta$ contains symmetric rank 2 hyperbolic root subsystems of type $H(k,k)$ for infinitely many distinct $k\geq 3$. 
If $\Delta$ is non-symmetric of type $H(a,b)$, then it
also contains non-symmetric rank 2 hyperbolic root subsystems of type $H(a\ell,b\ell)$ for infinitely many distinct $\ell\geq 2$.
\end{thm}

\begin{proof} The proof follows from the results in Table~\ref{T-Phisubs}. The first claim corresponds to cases $\text{II}_L$ and $\text{II}_S$. The second claim corresponds to case $\text{II}_{LS}$.
\end{proof}

Values of $\delta_d$ and $\epsilon_d$ for small $d$ are given in Table~\ref{T-subparams}.
\begin{table}

{\tiny
$$\begin{array}{r|rr}
d & \delta_d=\eta_d-\eta_{d-1} & \epsilon_d=\gamma_{d+1}-\gamma_{d} \\\hline
0 & & 1\\
1 & ab - 2 & ab - 3 \\
2 & a^2b^2 - 4ab + 2 & a^2b^2 - 5ab + 5 \\
3 & a^3b^3 - 6a^2b^2 + 9ab - 2 & a^3b^3 - 7a^2b^2 + 14ab - 7 \\
4 & a^4b^4 - 8a^3b^3 + 20a^2b^2 - 16ab + 2 & a^4b^4 - 9a^3b^3 +
27a^2b^2 - 30ab + 9 \\
5 & a^5b^5 - 10a^4b^4 + 35a^3b^3 - 50a^2b^2 + 25ab - 2 & a^5b^5 -
11a^4b^4 + 44a^3b^3 - 77a^2b^2 + 55ab - 11 \\
6 & a^6b^6 - 12a^5b^5 + 54a^4b^4 - 112a^3b^3 + 105a^2b^2 - 36ab + 2 &
a^6b^6 - 13a^5b^5 + 65a^4b^4 - 156a^3b^3 + 182a^2b^2 - 91ab + 13
\end{array}$$}
\caption{Values of $\delta_d$ and $\epsilon_d$ for small $d$}\label{T-subparams}
\end{table}

We can now prove Theorems~\ref{infshortrts} and~\ref{infbytype} for $\Phi$-subsystems. The latter is similar to results in \cite{CKMS} and \cite{Sr}.

\begin{thm}
Let $\Delta$ be a rank 2 root system and let $\Gamma$ be a nonempty set of real roots in $\Delta$.
\begin{enumerate}
\item If $\Delta$ is finite, then $\Phi(\Gamma)$ is finite.
\item If $\Delta$ is affine of type $\widetilde{A}_1$, then $\Phi(\Gamma)$ has finite type $A_1$ or affine type $\widetilde{A}_1$.
\item If $\Delta$ is affine of type $\widetilde{A}_2^{(2)}$, then $\Phi(\Gamma)$ has finite type $A_1$, or affine type $\widetilde{A}_1$ or $\widetilde{A}_2^{(2)}$.
\item If $\Delta$ is hyperbolic, then $\Phi(\Gamma)$ has finite type $A_1$ or hyperbolic type.
\end{enumerate}
\end{thm}

\begin{proof}
Part (i) is clear. The finite type $A_1$ occurs exactly when $\Gamma\subseteq\{\pm\alpha\}$, so we will assume from now on that this is not the case.

If $\Delta$ is affine, then $ab=4$ and it is easy to show from the recursion formulas in Lemma~\ref{acp} that
$\delta_d=\eta_d-\eta_{d-1}=2$ and $\epsilon_d=\gamma_{d+1}-\gamma_d=1$.
Parts (ii) and (iii) now follow.

If $\Delta$ hyperbolic, then $ab>4$, and so for $d>1$
$$\delta_d=\eta_d-\eta_{d-1} = (ab-2)\eta_{d-1} -\eta_{d-2}-\eta_{d-1} >2\eta_{d-1} -\eta_{d-2}-\eta_{d-1} =\eta_{d-1}-\eta_{d-2}=\delta_{d-1}.$$
By induction we get $\delta_d \ge \delta_1=(ab-1)-1=ab-2>2$ for all $d>0$.
It now follows that $H(\delta_d,,\delta_d)$ is hyperbolic since ${\delta_d}^2>4$.

A similar argument shows that
$\epsilon_d>\epsilon_{d-1}$ for $d>0$,
and so $\epsilon_d\ge\epsilon_0=1$ for $d\ge0$.
and so $H(a\epsilon_d,b\epsilon_d)$ is hyperbolic.
\end{proof}

\subsection{Classification of $\Delta$-subsystems in rank 2}
We now consider the classification of $\Delta$-subsystems of $\Delta$.
Let $\Gamma\subseteq \Delta^\re$ nonempty and recall that $\Delta(\Gamma)=\mathbb{Z}\Gamma\cap\Delta$,
 $\Delta^\re(\Gamma)=\mathbb{Z}\Gamma\cap\Delta^\re$.
Since the imaginary roots of an affine or hyperbolic root system are just the linear combinations of real roots
with nonpositive norm, it will suffice to describe $\Delta^\re(\Gamma)$.
From the definition of a reflection, we can see that $w_\a\realroots(\Gamma)\subseteq \realroots(\Gamma)$ for all $\a\in\Gamma$, and so
$$\Phi(\Gamma)\subseteq\realroots(\Gamma).$$

We also have $\Phi(\realroots(\Gamma))=\realroots(\Gamma)$, so the real roots of a $\Delta$-subsystem always form a $\Phi$-subsystem, but possibly for a different set of generators.

The classification of $\Delta$ subsystems reduces to divisibility properties for the sequences $\eta_j$ and $\gamma_j$. 
This lemma requires a complex but elementary induction argument:
\begin{lemma}\label{divrec}  (\cite{CKMS},\cite{Sr}) Let $a\ge b\ge1$ with $ab\ge4$, and let $d\ge0$, $i\in\mathbb{Z}$. Then
\begin{align*}
\gamma_d\delta_{j-d}&=\gamma_j-\gamma_{j-2d},\\
\eta_d\epsilon_{j-d-1}&=\gamma_j-\gamma_{j-2d-1},\\
\eta_d\delta_{j-d}&=\eta_j-\eta_{j-2d-1},\\ 
ab\gamma_d\epsilon_{j-d}&=\eta_j-\eta_{j-2d}.
\end{align*}
\end{lemma}

\begin{lemma} \label{div} Let $a\ge b\ge1$ with $ab\ge4$, and let $d\ge0$, $j\in\mathbb{Z}$.
\begin{enumerate}
\item\label{div-ab} $\gcd(a,\eta_j)=\gcd(b,\eta_j)=1$. 
\item\label{div-gg} $\gamma_d \mid \gamma_j$ if and only if $j\in d\mathbb{Z}$.
\item\label{div-eg} $\eta_d\mid \gamma_j$ if and only if $j\in(2d+1)\mathbb{Z}$. 
\item\label{div-ee} $\eta_d\mid\eta_j$ if and only if $j\in d+(2d+1)\mathbb{Z}$.
\item\label{div-ge} $\gamma_d\mid\eta_j$ if and only if $d=1$, when $ab>4$.
\item\label{div-ge41} $\gamma_d\mid\eta_j$ if and only if $d=2e+1$ is odd and $j\in e+(2e+1)\mathbb{Z}$, when $ab=4$.
\end{enumerate}
\end{lemma}
\begin{proof}
\ref{div-ab} This follows from the fact that $\eta_j\equiv(-1)^j\pmod{ab}$, which is easily proved by induction.

Cases (ii)-(v) proceed by repeated application of Lemma~\ref{divrec}. Part (v) also uses Lemma~\ref{acp}. Part (vi) is immediate.
\end{proof}

\newpage

\begin{thm}
\label{subsysbeq1}
Let $\Delta$ be a rank 2 root system of type $H(a,b)$ with $a\ge b$ and $ab\ge4$.
Let $\Gamma\subseteq\Delta^\re$ be nonempty. 
\begin{enumerate}
\item If $a>4$, $b=1$ and $\Phi(\Gamma)$ is the subsystem consisting of all short roots in $\Delta^\re$, then $\Delta^\re(\Gamma)=\Delta^\re\ne\Phi(\Gamma)$.

\item If $a=4$, $b=1$ and $\Phi(\Gamma)$ is a subsystem of type $\text{II}_S$ with basis $\a^{SU}_r,\a^{SL}_{d-r-1}$ for some odd $d=2e+1$ and $0\le r<d$,
then $\Delta^\re(\Gamma)\ne\Phi(\Gamma)$ is a subsystem of type $\text{II}_{LS}$ with basis $\a^{LL}_{s},\a^{SU}_{e-s}$ where $s\equiv e-r\pmod{d}$ and $-e\le s\le e$.  
 
\item In all other cases,  $\Delta^\re(\Gamma)=\Phi(\Gamma)$.
\end{enumerate}
\end{thm}

\begin{proof} 
If $\Phi(\Gamma)$ has type $\text{I}_L$ or $\text{I}_S$, then it is clear that  $\Delta^\re(\Gamma)=\Phi(\Gamma)$.\\

Suppose $\Phi(\Gamma)$ has type $\text{II}_L$. Since $\Phi(\Gamma)=(w_1w_2)^r\Phi(\{\a^{LL}_0,\a^{LU}_{d-1}\})$, it suffices to consider $r=0$.\\

Now $\a^{LL}_0=\a_1$ and $\a^{LU}_{d-1}=\eta_{d-1}\a_1+a\gamma_d\a_2$, so
$$\Delta^\re(\Gamma) =  \mathbb{Z}\{\alpha^{LL}_0,\alpha^{LU}_{d-1}\} \cap \Delta^\re = \mathbb{Z}\{\alpha_1,a\gamma_d\alpha_2\}\cap \Delta^\re.$$

The root $\a^{LL}_j=\eta_{j}\alpha_1+a\gamma_{j}\alpha_2$ is in $\Delta^\re(\Gamma)$ if and only if $\gamma_d\mid\gamma_j$ if and only if $j\in d\mathbb{Z}$ by Lemma~\ref{div}\ref{div-gg}.\\

Also $\a^{SU}_j=b\gamma_j\a_1+\eta_i\a_2$ is in $\Delta^\re(\Gamma)$ if and only if $a\gamma_d\mid \eta_j$ which is not possible by Lemma~\ref{div}\ref{div-ab} since $a>1$.
Hence $\Delta^\re(\Gamma)=\Phi(\Gamma)$.\\

Suppose $\Phi(\Gamma)$ has type $\text{II}_{LS}$. Since $\Phi(\Gamma)=(w_2w_1)^{d-r}\Phi(\{\a^{LL}_d,\a^{SU}_{0}\})$, it suffices to consider $r=d$.

We have $\a^{LL}_{d}=\eta_d\a_1+a\eta_d\a_2$ and $\a^{SU}_0=\a_2$, so
$$\Delta^\re(\Gamma) =  \mathbb{Z}\{\alpha^{LL}_d,\alpha^{SU}_{0}\} \cap \Delta^\re = \mathbb{Z}\{\eta_d\alpha_1,\alpha_2\}\cap \Delta^\re.$$

Now $\a^{LL}_j=\eta_{j}\alpha_1+a\gamma_{j}\alpha_2$ is in $\Delta^\re(\Gamma)$ if and only if $\eta_d\mid\eta_j$  if and only if $j\in d+(2d+1)\mathbb{Z}$ by Lemma~\ref{div}\ref{div-ee}.
And $\a^{SU}_j=b\gamma_j\a_1+\eta_j\a_2$ is in $\Delta^\re(\Gamma)$ if and only if $\eta_d\mid b\gamma_i$ if and only if $j\in (2d+1)\mathbb{Z}$ by Lemma~\ref{div}\ref{div-eg}.
Hence $\Delta^\re(\Gamma)=\Phi(\Gamma)$.

Finally suppose $\Phi(A)$ has type $\text{II}_S$. Since $\Phi(\Gamma)=(w_2w_1)^r\Phi(\{\a^{SU}_0,\a^{SL}_{d-1}\})$, it suffices to consider $r=0$.

Now $\a^{SU}_0=\a_2$ and $\a^{SL}_{d-1}=b\gamma_d\a_1+\eta_{d-1}\a_2$, so
$$\Delta^\re(\Gamma) =  \mathbb{Z}\{\alpha^{LL}_0,\alpha^{LU}_{d-1}\} \cap \Delta^\re = \mathbb{Z}\{b\gamma_d\alpha_1,\alpha_2\}\cap \Delta^\re.$$

We have $\a^{SU}_j=b\gamma_i\a_1+\eta_j\a_2\in \Delta^\re(\Gamma)$ if and only if $\gamma_d\mid \gamma_j$ if and only if $j\in d\mathbb{Z}$.
And $\a^{LL}_j=\eta_{j}\alpha_1+a\gamma_{j}\alpha_2$ is in $\Delta^\re(\Gamma)$ if and only if $b\gamma_d\mid\eta_j$.
By Lemma~\ref{div}\ref{div-ab}, \ref{div-ge}, and \ref{div-ge41},  this can only happen if $a>4$, $b=1$ and $d=1$; or $a=4$, $b=1$ and $d$ odd.\\

If $a>4$, $b=1$, and $d=1$, then $\Phi(\Gamma)$ is the set of all short real roots, and $b\gamma_d\mid\eta_j$ for all $j$ so $\Delta^\re(\Gamma)=\Delta^\re$.\\

If $a=4$, $b=1$, and $d=2e+1$, then $b\gamma_d\mid\eta_i$ if and only if $j\in e+(2e+1)\mathbb{Z}$, so $I^S=(2e+1)\mathbb{Z}$, $I^L=e+(2e+1)\mathbb{Z}$, and hence $\Delta^\re(\Gamma)$ has type $\text{II}_{LS}$ with the given basis. In all other cases  $\Delta^\re(\Gamma)=\Phi(\Gamma)$.
\end{proof}

We may restate Theorem~\ref{subsysbeq1} in the following way.
\begin{thm}\label{restateinfshortrts}
Let $\Delta$ be a  rank 2 infinite root system and let $\Gamma$ be a set of real roots which generate $\Delta$, that is, $\Delta(\Gamma)=\Delta$.
Then either $\Phi(\Gamma)$ is the set of all real roots in $\Delta$ or it is the set of all \emph{short} real roots in $\Delta$.
The second case occurs only if $a=1$ or $b=1$ and $\Phi(\Gamma)$ consists of short roots.
\end{thm}

Theorems~\ref{infshortrts} and~\ref{infbytype} now follow for $\Delta$-subsystems.

\section{Structure constants}

Let $A =(a_{ij})_{i,j\in I}$ for $I=\{1,\dots ,\ell\}$ be any  symmetrizable generalized Cartan matrix.
Let $\mathfrak{g}=\mathfrak{g}(A)$ be the Kac--Moody algebra corresponding to  $A$ with Cartan subalgebra $\mathfrak{h}$.
Let $\langle\cdot,\cdot\rangle: \mathfrak{h}^{\ast}\longrightarrow \mathfrak{h}$ denote the natural nondegenerate bilinear pairing between  $\mathfrak{h}$ and its dual $\mathfrak{h}^{\ast}$.
Let $\Delta=\Delta(A)$ be the corresponding root system with simple roots $\Pi=\{\alpha_1,\dots,\alpha_{\ell}\}\subseteq \mathfrak{h}^{\ast}$ and simple coroots $\Pi^{\vee}=\{\alpha_1^{\vee},\dots,\alpha_{\ell}^{\vee}\} \subseteq
\mathfrak{h}$ such that  
$$\langle\alpha_j,\alpha_i^{\vee}\rangle=\alpha_j(\alpha_i^{\vee})=a_{ij}$$ for $i,j\in I$.

Then $\mathfrak{g}$ has root space decomposition
$$\mathfrak{g}=\mathfrak{g}^+\ \oplus\ \frak{h}\ \oplus\ \mathfrak{g}^-,$$  
$$\mathfrak{g}^+ =\bigoplus_{\alpha\in\Delta^+}\mathfrak{g}_{\alpha},\qquad
\mathfrak{g}^- = \bigoplus_{\alpha\in\Delta^-}\mathfrak{g}_{\alpha}.$$ 
The Chevalley involution $\omega$ is an automorphism of $\mathfrak{g}$ with $\omega^2=1$ and $\omega(\mathfrak{g}_\a)=\mathfrak{g}_{-\a}$ for all $\a\in\Delta$.

In a future paper we will construct Chevalley bases for these algebras, but in the following we restrict our attention to the real root spaces.
Recall that if $\alpha\in\Delta^{\re}$, then the root space $\mathfrak{g}_{\alpha}$ is one dimensional. 
For each $\a\in\Delta^{\re}_+$ we  a choose root vector $x_{\a}\in\mathfrak{g}_{\alpha}$ and  $x_{-\a}:=-\omega(x_\a)$, normalized
so that $[x_{\a},x_{-\a}] = \a^{\vee}$.
In the remaining sections,  our choices  of root vectors $\{x_{\a}\mid\a\in\Delta^\re\}$ are assumed to satisfy
\begin{align}
\omega(x_\a)&=-x_{-\a}&&\text{for $\a\in\Delta^{\re}$},\label{rel-Chev}\\
[x_{\a},x_{-\a}] &= \a^{\vee}  &&\text{for $\a\in\Delta^{\re}$},\label{rel-neg}\\
[x_{\a},x_{\b}]&= 0 &&\text{for $\a,\b\in\Delta^{\re}$ with  $\a+\b\notin\Delta\cup\{0\}$},\label{rel-0}\\
[x_{\a},x_{\b}]&= n_{\a,\b} \ x_{\a+\b} &&\text{for $\a,\b\in\Delta^{\re}$ with $\a+\b\in\Delta^{\re}$ and  $x_{\a+\b}\in \mathfrak{g}_{\alpha+\beta}$}.
\label{rel-triple}
\end{align}

Given roots $\a,\b\in\Delta$, the   $\a$--root string though $\b$  is:
$$\b-p_{\a,\b}\a,\dots,\b-\a,\b, \b+\a,\dots,\b+q_{\a,\b}\a.$$
The following result  is proven in analogy with the finite dimensional case (noted  also in \cite{Mor}). For rank 2 symmetrizable Kac--Moody algebras,   this was proven in detail in \cite{CCFMMTZ}, following (\cite{Hu1}, Ch 25).
\begin{lemma}\label{L-nab} 
Let $\omega$ be the Chevalley involution and suppose that $\omega(x_{\a}) = -x_{-\a}$. Then
$$n_{\a,\b}=\pm(p_{\a,\b}+1).$$
for all $\a,\b,\a+\b\in\Delta^\re$.
\end{lemma}
We  write 
$$n_{\a,\beta}=s_{\a,\beta}(p_{\a,\beta}+1)$$
for some $s_{\a,\b}\in\{\pm1\}$. We refer to $s_{\a,\b}$ as the {\it sign} of the structure constant $n_{\a,\beta}$.
Note that  $p_{\a,\b}$ is determined by the root strings in the root system, but the possible systems of signs
$\{s_{\a,\b}\mid\a,\b,\a+\b\in\Delta^\re\}$ are harder to determine.

The following result generalizes the well known analogous result in the finite dimensional case, proven in Theorem 4.1.2 of [Ca] (see also Proposition 8.1 of \cite{Cas1}). For symmetrizable Kac--Moody algebras,  this theorem is proven in detail in \cite{CCFMMTZ} and (ii) was noted in \cite{Cas2}.

\begin{theorem}\label{lengths} The structure constants of a symmetrizable Kac--Moody algebra satisfy:
\begin{enumerate}
\item\label{len-swap}$n_{\a,\beta}=-n_{\beta,\a}=-n_{-\a,-\b}$ for all $\a,\b,\a+\b\in\Delta^{\re}$.
\item\label{len-triple} If $\a,\b,\g\in\Delta^{\re}$ and $\a+\beta+\gamma=0$, then 
$$\dfrac{n_{\a,\beta}}{| \gamma |^2}=\dfrac{n_{\beta,\gamma}}{|\a|^2}=\dfrac{n_{\gamma,\a}}{|\b|^2}.$$ 
\item\label{len-quadruple} Let $\a,\b,\g,\d\in\Delta^{\re}$ and $\a+\beta+\gamma+\d=0$ with no two roots antipodal.  If every pairwise sum of $\a,\b,\g,\d$ is a real root, then we have:
$$\dfrac{n_{\a,\beta}n_{\g,\d}}{|\a+\b|^2}+\dfrac{n_{\beta,\gamma}n_{\a,\d}}{|\b+\g|^2}+
\dfrac{n_{\gamma,\a}n_{\b,\d}}{|\g+\a|^2}=0.$$ 
\end{enumerate}
\end{theorem}
\begin{proof} The first equality of (i) follows from skew symmetry. For the second equality of (i), if $\a+\b\in\Delta^{\re}$ then $-\a-\b\in\Delta^{\re}$ and we have
$$n_{-\a,-\beta} x_{-\a-\beta} =[x_{-\a},x_{-\beta} ]=[\omega(x_{\a}),\omega(x_{\beta})]=\omega[x_{\a},x_{\beta}]$$
$$=\omega(n_{\a,\beta} x_{\a+\beta})=-n_{\a,\beta} x_{-\a-\beta}.$$
Hence $n_{-\a,-\beta} =-n_{\a,\beta}.$

For (ii),  since $\a+\beta+\gamma=0$, we have
$$ \a+\beta=-\gamma,\ \beta+\gamma=-\a,\ \a+\gamma=-\beta.$$
By the Jacobi identity
$$[[x_{\a},x_{\b}],x_{\g}]+[[x_{\b},x_{\g}],x_{\a}]+[[x_{\g},x_{\a}],x_{\b}]=0.$$
Thus
$$n_{\a,\beta}[x_{-\g},x_{\g}]+n_{\beta,\gamma}[x_{-\a},x_{\a}]+n_{\gamma,\a}[x_{-\b},x_{\b}]=0$$
So
$$-n_{\a,\beta}\g^\vee-n_{\beta,\gamma}\a^\vee-n_{\gamma,\a}\b^\vee=0.$$
The proof follows using the identifications $\g^\vee=\dfrac{2\g}{(\g,\g)}$, $\a^\vee=\dfrac{2\a}{(\a,\a)}$ and $\b^\vee=\dfrac{2\b}{(\b,\b)}$.

For (iii), let $\a,\b,\g,\d\in\Delta^\re$ with $\a+\beta+\gamma+\d=0$ and suppose that every pairwise sum of $\a,\b,\g,\d$ is a real root.  Then we have
\begin{align*}
    \alpha + \beta & = -\gamma - \delta\\
    \beta + \gamma & = -\alpha - \delta\\
    \gamma + \alpha & = -\beta - \delta
\end{align*}
The proof follows by applying \(\text{ad}x_{\delta}\) to the Jacobi identity to obtain
$$0=[[[x_{\alpha}, x_{\beta}], x_{\gamma}],x_{\delta}] + [[[x_{\beta}, x_{\gamma}], x_{\alpha}],x_{\delta}] + [[[x_{\gamma}, x_{\alpha}], x_{\beta}],x_{\delta}]$$
and using (ii).
\end{proof}

We now generalize Carter's method (\cite{Ca1}) of finding structure constants from signs on extraspecial pairs. We apply this method to rank 2 Kac-Moody algebras in the next two sections.

Consider a subset $E\subseteq \Delta^\re_+\times \Delta^\re_+$ and $(\a,\b)\in (\Delta^\re_+\times \Delta^\re_+)-E$.
We say that the sign on $(\a,\b)$ is \emph{determined} by the signs on $E$ if
$$ s_{\a',\b'}=\bar{s}_{\a',\b'}\text{ for all $(\a',\b')\in E$} \quad\implies\quad s_{\a,\b}=\bar{s}_{\a,\b},$$
where $\{x_\a\mid\a\in\Delta^\re_+\}$ and $\{\bar{x}_\a\mid\a\in\Delta^\re_+\}$  are systems of root vectors satisfying (\ref{rel-Chev})--(\ref{rel-triple}) with corresponding  signs
 $\{s_{\a,\b}\mid \a,\b,\a+\b\in\Delta^\re_+\}$ and  $\{\bar{s}_{\a,\b}\mid \a,\b,\a+\b\in\Delta^\re_+\}$ respectively.
 
Take a well-ordering $<$ on the set $\Delta_+^\re$ that respects addition, that is, $\a<\a+\b$ and $\b<\a+\b$ whenever $\a,\b,\a+\b\in\Delta_+^\re$.
We then define the set of \emph{special pairs}
$$P=\{(\a,\b)\mid \a,\b,\a+\b\in\Delta_+^\re\text{ with } \a<\b\}.$$
Take a special pair $(\a,\b)$ and set $\gamma=-\a-\b$.
By Lemma~\ref{lengths}\ref{len-swap} and~\ref{len-triple}, $s_{\a,\b}$ determines all structure constants involving the triples $\{\a,\b,\g\}$ and $\{-\a,-\b,-\g\}$ as follows:
\begin{align*}
s_{\a,\b}&=s_{-\b,-\a}=s_{-\a,-\g}=s_{\g,\a}=s_{-\g,-\b}=s_{\b,\g},\\
-s_{\a,\b}=s_{\b,\a}&=s_{-\a,-\b}=s_{-\g,-\a}=s_{\a,\g}=s_{-\b,-\g}=s_{\g,\b}.
\end{align*}
The entire system of signs of structure constants is now determined by the signs on the  special pairs $P$, using Algorithm~1 of~\cite{CMT}.
 
The ordering $<$ gives the usual lexicographic ordering on pairs, that is,
$(\a,\b)<(\a',\b')$ whenever $\a<\a'$, or $\a=\a'$ and $\b<\b'$.
Since this is a well-ordering, 
we can define the \emph{extraspecial pairs} by transfinite induction as follows: Suppose $(\a,\b)$ is a special pair and $E$ is the set of all extraspecial pairs less than $(\a,\b)$. Then we define $(\a,\b)$ to be extraspecial if, and only if, the sign on $(\a,\b)$ is \emph{not} determined by the signs on $E$.
It is clear that  the signs of the structure constants on extraspecial pairs determine all structure constant signs involving real root vectors. 

Define a positive root to be \emph{decomposable} if it can be written as a sum of two other positive roots. 
In finite root systems, a positive root is  decomposable if, and only if, it is nonsimple.
Lemma~\ref{lengths}\ref{len-quadruple} ensures that, in the finite-dimensional case, there is exactly one extraspecial pair $(\a,\b)$ with $\a+\b=\gamma$ for each decomposable positive root $\gamma$.

Lemma~\ref{lengths}\ref{len-quadruple} is less useful in the infinite dimensional case, and we need a different approach to show that
this property still holds for extraspecial pairs of rank 2 Kac-Moody algebras.

\section{Structure constants for rank 2 Kac--Moody algebras, except $A^{(2)}_2$}
In this section we completely determine the structure constants involving real root vectors in the rank 2 Kac--Moody algebras $H(a,b)$ for $ab>4$ and $H(2,2)=A_1^{(1)}$. 

For convenience, in this section and the following, we write $x_j^{LU}$ instead of $x_{\a_J^{LU}}$, and so on.

\begin{proposition}\label{P-allsc-triv}
If $\Delta=\Delta(H(a,b))$ with $a,b>1$ and $ab\geq 4$, then there are no triples $\a,\b,\g\in\Delta^\re$ with $\a+\b+\g=0$. 
\end{proposition}
\begin{proof} Suppose that there exist $\a,\b,\g\in\Delta^\re$ with $\a+\b+\g=0$. Then $\a+\b=-\g$ is a real root.

But by Proposition~\ref{sumbne1}, when $a,b>1$ and $ab\geq 4$, no sum of real roots can be a real root. 
\end{proof}
Theorem~\ref{T-sc}\ref{T-sc-triv} now follows.

\begin{proposition}\label{triples}
Let $\Delta=\Delta(H(a,1))$ with $a>4$. Then $\a,\b,\g$ are real roots of $\Delta$ with $\a+\b+\g=0$ if and only if, for some $j\in\Z$,
$$\{\a,\b,\g\}= \{\a^{SU}_j,\a^{SU}_{j+1},-\a^{LU}_j\}\quad\text{or}\quad\{\a^{SL}_j,\a^{SL}_{j+1},-\a^{LL}_{j+1}\}.$$
\end{proposition}

\begin{proof} We have real roots $\a,\b,\g$ satisfying $\a+\b+\g=0$ if and only if
$\a+\b=-\g.$ 
By Proposition~\ref{sums} two of the three roots must be short so, without loss of generality, assume $\a,\b$ are short roots.
Then there is a $w\in W$ such that $w^{-1}\a=\a_2$. Since $w^{-1}\a+w^{-1}\b=-w^{-1}\g$ is a real root, Theorem~\ref{T-beq1} shows that $w^{-1}\b$ is s $\a_1$ or $\a_1+(a-1)\a_2$. But $\a_1$ is long, so $w^{-1}\b=\a_1+(a-1)\a_2$.
We get the triples $\{\a^{SU}_j,\a^{SU}_{j+1},-\a^{LU}_j\}$ and
$\{\a^{SL}_j,\a^{SL}_{j+1},-\a^{LL}_{j+1}\}$ by taking $w=(w_1w_2)^j$ and $(w_1w_2)^jw_1$, respectively.
\end{proof}

\begin{theorem}
Let $\Delta=\Delta(H(a,1))$ with $a>4$. 
Let $H\in\{U,L\}$ and set $\a=\a^{SH}_j$, $\b=\a^{SH}_{j+1}$, $\g=-(\a+\b)$, then
\begin{alignat*}{3}
 p_{\a,\b}=p_{\b,\a}&=a-1,\qquad & p_{\a,\g}=p_{\g,\a}= p_{\b,\g}=p_{\g,\b}&=0.
 \end{alignat*}
\end{theorem}

\begin{proof} 
As in the proof of Proposition~\ref{triples}, we can find $w\in W$ such that
$$w^{-1}\a=\a_2,\quad w^{-1}\b=\a_1+(a-1)\a_2,\quad w^{-1}\g=-\a_1-a\a_2.$$

Now $p_{\a,\g}=0$ since $w^{-1}(\g-\a)=-\a_1-(a+1)\a_2$ is not a root
and $p_{\b,\g}=0$ since $w^{-1}(\g-\b)=-a\a_2$ is not a root.

Finally   the $w^{-1}\a$ root string through $w^{-1}\b$ is
$$w^{-1}\b-(a-1)w^{-1}\a= \a_1,\  \a_1+\a_2,\  \a_1+2\a_2,\ \dots,\ w^{-1}\b=\a_1+(a-1)\a_2, \ w^{-1}(\b+\a)=\a_1+a\a_2,$$
and thus $p_{\a,\b}=p_{w^{-1}\a,w^{-1}\b}=a-1$.
\end{proof}

We can now determine all possible systems of signs for these structure constants. We take our ordering on $\Delta^\re_+$ to be
\begin{multline*}
a^{SU}_0<\a^{LL}_0<\a^{SL}_{0}<\a^{SU}_{1}< \a^{SL}_{1}<\a^{LU}_0<\a^{LL}_1<\cdots\\
\cdots < \a^{SU}_{j}<\a^{SL}_j<\a^{SU}_{j+1}<\a^{SL}_{j+1}<\a^{LU}_j<\a^{LL}_{j+1}<\cdots,
\end{multline*}
for $j\ge1$.
Now the set of decomposable roots is
$$\{\a^{SL}_0\}\cup\{\a_j^{LU},\a_{j+1}^{LL}\mid j\ge0\},$$
and the set of special pairs is
$$P= \{(\a^{SU}_0,\a^{LL}_{0})\}\cup\{(\a^{SB}_j,\a^{SB}_{j+1})\mid j\ge0,\; B\in\{U,L\}\},$$
with exactly one special pair corresponding to each decomposable root.
We now show that every special pair is an extraspecial pair:
\begin{theorem}\label{T-Ha1signs}
Let $A=H(a,1)$ with $a>4$ and let $\mathfrak{g}=\mathfrak{g}(A)$. 
Then the extraspecial pairs for $\mathfrak{g}$ are $(\a^{SU}_0,\a^{LL}_{0})$ with sum $\a^{SL}_0$; and, for each integer $j\ge0$,
$(\a^{SU}_j,\a^{SU}_{j+1})$ with sum $\a_j^{LU}$ and $(\a^{SL}_j,\a^{SL}_{j+1})$ with sum $\a_{j+1}^{LL}$.
\end{theorem}
\begin{proof}
For all $j\ge0$ define signs for these pairs  as follows:
$$s^{UL}:= s_{\a^{SU}_0,\a^{LL}_{0}},\quad s^U_j:=s_{\a^{SU}_j,\a^{SU}_{j+1}}\quad\text\and\quad s^L_j:=s_{\a^{SL}_j,\a^{SL}_{j+1}}.$$
Suppose $\a=\a^{SU}_j$, $\b=\a^{SU}_{j+1}$, $\g=-\a^{LU}_j=-\a^{SU}_j-\a^{SU}_{j+1}$.
We claim that we can make any choice of signs $s^{UL},s^U_j,s^L_j\in\{+1,-1\}$ for $j\ge0$.
This can be seen by varying the basis we use for $\mathfrak{g}$. 
Take $\a=\a^{SU}_j$, $\b=\a^{SU}_{j+1}$, $\g=-\a^{LU}_j=-\a^{SU}_j-\a^{SU}_{j+1}$ for $j\in\N$.
Suppose we replace root vector $x^{SU}_k$ with $-x^{SU}_k$ for $k\le j$;
replace root vector $x^{SL}_k$ with $-x^{SL}_k$ for $k\ge -j-1$;
and leave all other root vectors unchanged.
Clearly this will not impact properties (\ref{rel-Chev})--(\ref{rel-0}).
Now for $\ell>0$, the Lie algebra relation
$$[x^{SL}_\ell,x^{SL}_{\ell+1}] = s^L_\ell ax^{LU}_\ell$$ 
will be unaffected since $\ell>0\ge -j-1$.
Now consider the relation 
$$[x^{SU}_\ell,x^{SU}_{\ell+1}] = s^U_\ell ax^{LU}_\ell .$$
For $\ell>j$, this relation is also unaffected;
for $0<\ell<j$, two of the three root vectors swap their signs, so the value of $s^U_\ell$ is unchanged;
for $\ell=j$ only one of the three root vectors swaps its sign, so $s^U_\ell$ is replaced by $-s^U_\ell$.
Since all other relations of the form (\ref{rel-triple}) follow from the values of $s^U_\ell,s^L_\ell$ for $\ell>0$, this has the effect of swapping 
a single sign $s^U_j$. A similar argument allows us to swap a single sign $s^L_\ell$,
and we can swap the sign of $s^{UL}$ by swapping the sign of $x^{LL}_0$ and $x^{LU}_{-1}$ without affect any of the other signs.
By transfinite induction, we can now choose these signs arbitrarily.
\end{proof}
Theorem~\ref{T-sc}\ref{T-sc-Ha1} is an immediate corollary of this result.

\section{Structure constants for the twisted affine algebra of type $A_2^{(2)}=H(4,1)$}

In [L], Lepowsky  constructed twisted affine Lie algebras using twisted vertex operators, giving  a  description of structure constants for twisted affine algebras.  In recent work, Calinescu, Lepowsky and Milas (\cite{CLM}) gave an explicit construction of the rank 2 twisted affine algebra $A_2^{(2)}$, including implicit structure constants (up to signs),  using the general method of Lepowsky. 

In this section, we compute constants arising from root strings and construct all possible systems of signs.

\begin{figure}[h]
\begin{center}
{\includegraphics[width=3.6in]{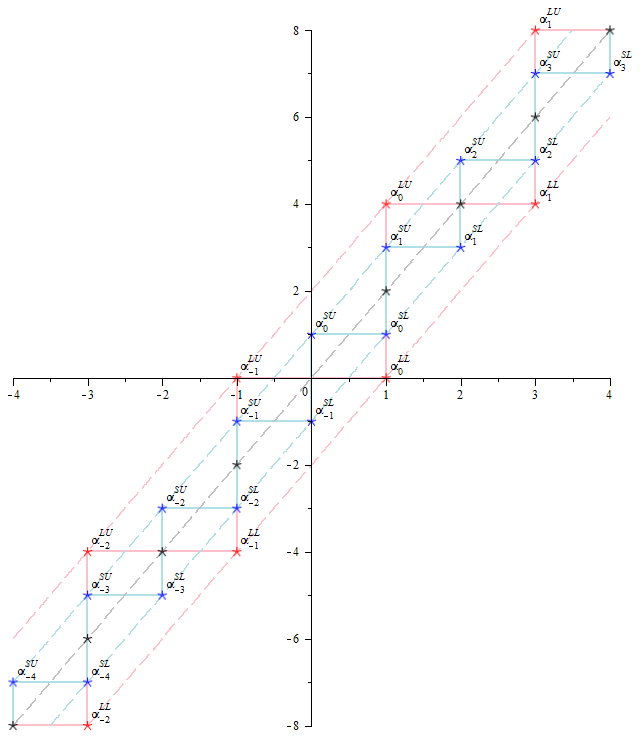}}
\caption{The root system of $A_2^{(2)}=H(4,1)$}\label{Hyp}
\end{center}
\end{figure}

\begin{theorem}
Let $\Delta=\Delta(H(4,1))$. Then $\a,\b,\g$ are real roots of $\Delta$ with $\a+\b+\g=0$ if and only if
$$\{\a,\b,\g\}= \{\a^{SU}_j,\a^{SU}_{2k+1-j},-\a^{LU}_k\}\quad\text{or}\quad\{\a^{SL}_j,\a^{SL}_{2k+1-j},-\a^{LL}_{k+1}\},$$
for some $j,k\in\Z$. Furthermore, if $H\in\{U,L\}$ and we set $\a=\a^{SH}_j$, $\b=\a^{SH}_{2k+1-j}$, $\g=-(\a+\b)$, then
\begin{alignat*}{3}
 p_{\a\b}&=3,\qquad & p_{\a\g}&=0, \qquad& p_{\b\g}&=0,\\
  p_{\b\a}&=3, & p_{\g\a}&=0, & p_{\g\b}&=0.
\end{alignat*}
That is, if  $\a$ and $\b$ are short roots whose sum is long, then $p_{\a\b}=3$ otherwise $p_{\a\b}=0$.
\end{theorem}
\begin{proof}
This is elementary using
\begin{align*}
\alpha_j^{LL}&= (2j+1)\alpha_1+4j\alpha_2,& 
\alpha_j^{LU}&= (2j+1)\alpha_1+4(j+1)\alpha_2,\\
\alpha_j^{SU}&= j\alpha_1+(2j+1)\alpha_2,& 
\alpha_j^{SL}&= (j+1)\alpha_1+(2j+1)\alpha_2,
\end{align*}
and the imaginary roots are the nonzero integral multiples of $\a_1+2\a_2$, as shown in Figure~\ref{Hyp}.
\end{proof}

To find a system of structure constants, we first construct the Kac--Moody algebra of type $A_2^{(2)}=H(4,1)$ explicitly,
following \cite{Ca2} and \cite{Mitz}. Let $L^0$ be a Lie algebra of type $A_2$ with Lie bracket $[\cdot,\cdot]_0$, with  basis 
$$E_1,E_2,E_\theta=-[E_1,E_2]_0,\ F_1,F_2,F_\theta=[F_1,F_2]_0,\ H_1=[E_1,F_1]_0,H_2=[E_2,F_2]_0,$$
satisfying $[E_i,H_i]_0=-2E_i$ for $i=1,2$.
Fix invariant form 
$(\cdot,\cdot)_0$ with $( H_1+H_2,H_1+H_2)_0=2$.
Then define ${\mathcal{L}}(L^0) = \C[t,t^{-1}] \otimes L^0$, with cocycle $\kappa: {\mathcal{L}}(L^0)\times{\mathcal{L}}(L^0)\rightarrow \C$ given by 
$$\kappa(a,b) = \text{the coefficient of $t^{-1}$ in $\left( \frac{da}{dt},b\right)_0$}.$$
This defines a central extension 
$\widetilde{\mathcal{L}}(L^0) = {\mathcal{L}}(L^0) \oplus \C c $
with Lie product
$$[a+\lambda c, b+\mu c ] :=
  [a,b]_0 + \kappa(a,b)c .$$
The Kac--Moody Lie Algebra of type $\widetilde{A}^{(2)}_2$ is now the subalgebra $L\subseteq\widetilde{\mathcal{L}}(L^0)$ generated by
$1\otimes L^0$, $t\otimes E_\theta$, and $t^{-1}\otimes F_\theta$.

We can now give explicit coroots and root vectors of $L$ as follows: 
\begin{align*}
\a^\vee_1&=2\otimes(H_1+H_2), & \a^\vee_2&=-1\otimes(H_1+H_2) +c, \\
x^{SU}_r&=\sqrt{2}t^r\otimes\left(E_1+(-1)^rE_2\right), & x^{SL}_r&=\sqrt{2}t^{r+1}\otimes(F_1+(-1)^rF_2), \\
x^{LU}_r&=t^{2r+1}\otimes E_\theta, & x^{LL}_r&=t^{2r+1}\otimes F_\theta. 
\end{align*}
This gives explicit structure constants. For example,
\begin{align*}
[x^{SU}_j,x^{SU}_{2k+1-j}] &= \sqrt{2}t^j\sqrt{2}t^{2k+1-j}\otimes[E_1+(-1)^jE_2,E_1+(-1)^{2k+1-j}E_2]\\
&=2t^{2k+1}\otimes\left((-1)^{2k+1-j}[E_1,E_2]_0+(-1)^j[E_2,E_1]_0\right)\\
&=2t^{2k+1}(-(-1)^{2k+1-j}+(-1)^j)\otimes E_\theta =2t^{2k+1}2(-1)^{j}\otimes E_\theta  =4(-1)^{j}x^{LU}_k,
\end{align*}
and so $n_{\a^{SU}_j,\a^{SU}_{2k+1-j}} =  4(-1)^j.$
The signs of structure constants for our explicit root vectors are
\begin{align*}
s_{\a^{SB}_r,\a^{SB}_s}& =(-1)^r, &s_{\a^{SL}_i,\a^{LU}_j}&=(-1)^{i+1}, &s_{\a^{SU}_i,\a^{LL}_j}&=(-1)^{i+1}.
\end{align*}

We use the following ordering on $\Delta_+^\re$:
\begin{multline*}
\a^{SU}_0<\a^{SU}_1 <\a^{LL}_0<\a^{SL}_0< 
\a^{SL}_1<\a^{LU}_0<\a^{SU}_2<\a^{LL}_1<\a^{SL}_2<\cdots\\
\cdots<\a^{SU}_{2j-1}<\a^{SL}_{2j-1}<\a^{LU}_{j-1}<\a^{SU}_{2j}<\a^{LL}_j<\a^{SL}_{2j}<\cdots,
\end{multline*}
for $j>1$.
Now every positive real root is decomposable except for $\a^{LL}_0=\a_1$, $\a^{SU}_0=\a_2$, and $\a^{SU}_1$.
The special pairs are
\begin{align*}
P = \{ &(\a^{SU}_j,\a^{LL}_k)\mid 0\le j\le2k\text{ or $j=1$}\}\cup\{ (\a^{LL}_k,\a^{SU}_j)\mid j>2k\ge 0\text{ and $j\ne1$}\}\,\cup\\
 \{ &(\a^{SL}_j,\a^{LU}_k)\mid 0\le j\le2k+1\}\cup\{ (\a^{LU}_k,\a^{SL}_j)\mid j>2k+1\ge0\} \,\cup\\
\{ &(\a^{SB}_r,\a^{SB}_s) \mid  B\in\{U,L\}, \;0\le r<s,\; \text{$r+s$ odd}\}.
\end{align*}

The following theorem shows that there is exactly one extraspecial pair corresponding to each decomposable root.
\begin{theorem}\label{twisted}
Let $A=H(4,1)$ and let $\mathfrak{g}=\mathfrak{g}(A)$. 
Then the extraspecial pairs for $\mathfrak{g}$ are, for $k\ge0$,
\begin{align*}
(\a^{SU}_0,\a^{SU}_{2k+1}) &\text{ with sum $\a^{LU}_k$,}  
&(\a^{SL}_0,\a^{SL}_{2k+1}) &\text{ with sum $\a^{LL}_{k+1}$,}\\
(\a^{SU}_0,\a^{LL}_{k})&\text{ with sum $\a^{SL}_{2k}$,} &(\a^{SU}_1,\a^{LL}_{k})&\text{ with sum $\a^{SL}_{2k+1}$,} \\
(\a^{SL}_0,\a^{LU}_{k})&\text{ with sum $\a^{SU}_{2k+2}$,} &(\a^{SL}_1,\a^{LU}_{k})&\text{ with sum $\a^{SU}_{2k+3}$.}
\end{align*}
\end{theorem}

\begin{proof} 
Given $B\in\{U,L\}$, we define $b=0$ and $B'=L$ when $B=U$; $b=1$ and $B'=U$ when $B=L$. So
$$\a^{SB}_{2i}+\a^{SB}_{2j+1}=\a^{LB}_{b+i+j},\qquad
 \a^{SB'}_{2i+1}+\a^{LB}_{k-i}=\a^{SB}_{2(b+k)+1},\qquad
 \a^{SB'}_{2i}+\a^{LB}_{k-i}=\a^{SB}_{2(b+k)}.$$

The assumption that $\omega(x_{\a}) = -x_\a$ ensures that every possible system of signs of structure constants on real root vectors
can be achieved by  
 swapping the signs of the basis elements $x_\a$ for $\a\in\Delta^\re_+$,  provided that we do the same sign swaps to  $x_{-\a}$ for $\a\in\Delta^\re_+$.  
Let $\mathcal{R}$ be the set of all functions $\Delta^\re_+\rightarrow\Z_2$.
Given $\rho\in \mathcal{R}$, we  write $r^{HB}_j:=(-1)^{\rho(\a^{HB}_j)}$ and define  new root vectors
\begin{align*}
\bar{x}^{HU}_j&=r^{HU}_jx^{HU}_j, &\bar{x}^{HL}_j&=r^{HL}_jx^{HL}_j, &
\bar{x}^{HL}_{-j-1}&=r^{HU}_jx^{HL}_{-j-1},& \bar{x}^{HU}_{-j-1}&=r^{HL}_jx^{HU}_{-j-1},
\end{align*}
for $j\ge0$, $H\in\{S,L\}$,  $B\in\{U,L\}$.

Rather than use the special pairs $P$, we use a modified set
$$
P' = \{ (\a^{SU}_j,\a^{LL}_k)\mid j,k\ge0\}\cup
 \{ (\a^{SL}_j,\a^{LU}_k)\mid j,k\ge0\}\cup\{ (\a^{SB}_r,\a^{SB}_s) \mid  B\in\{U,L\}, \;0\le r<s,\; \text{$r+s$ odd}\}.
$$
This modification doesn't change the proof, apart from swapping the relevant signs, none of which are extraspecial.

Let $\mathcal{S}$ be the set of all functions $P'\rightarrow\Z_2$, and define a linear map
$\phi:\mathcal{R}\rightarrow\mathcal{S},\rho\mapsto \bar{\rho}$ by
$$\bar{\rho}(\a,\b) = \rho(\a)+\rho(\b) + \rho(\a+\b)$$ 
for all $(\a,\b)\in P'$.
Then $\bar{\rho}$ determines the signs of structure constants  as follows:
\begin{align*}
\bar{s}_{\a^{SB}_r,\a^{SB}_s}& =(-1)^{r+\bar{\rho}(\a^{SB}_r,\a^{SB}_s)}, &\bar{s}_{\a^{SL}_i,\a^{LU}_j}&=(-1)^{i+1+\rho(\a^{SL}_i,\a^{LU}_j)}, &\bar{s}_{\a^{SU}_i,\a^{LL}_j}&=(-1)^{i+1+\rho(\a^{SU}_i,\a^{LL}_j)}.
\end{align*}

Define the following elements of $\mathcal{R}$:
\begin{align*}
\lambda^S(\a)=1&\;\text{ if and only if }\;\a\in\{\a_{j}^{SB}\mid B\in\{U,L\},\; j\ge0\};\\
\lambda^L(\a)=1&\;\text{ if and only if }\;\a\in\{\a_{j}^{LB},\a_{2j+b}^{SB}\mid B\in\{U,L\},\; j\ge0\};\\
\lambda_1(\a)=1&\;\text{ if and only if }\;\a\in\{\a_{2j+1}^{LL},\a_{2j}^{LU},\a_{4j+1}^{SB},\a_{4j+2}^{SB}\mid B\in\{U,L\},\; j\ge0\}.
\end{align*}
Let $X=\{\lambda^S,\lambda^L,\lambda_1\}$. 
It is easily shown that each element of $X$ maps to $0$ under $\phi$,
so $\vspan(X)\subseteq\ker(\phi)$.
If we restrict the elements of $X$ to the set of roots $Y=\{\a^{LL}_0, \a^{SU}_0, \a^{SU}_1\}$,
they are linearly independent.
Hence $X$ is linearly independent and $\mathcal{R}=\vspan(X)\oplus\mathcal{R}'$, where
$\mathcal{R}'$ consists of all $\rho\in\mathcal{R}$ that take the value $0$ on every element of $Y$.

Suppose now that we are given arbitrary signs
$s^{LB}_{j+b}$ and $s^{SB'}_{j+2b}$, for $B\in\{U,L\}$,  $j\ge0$  (that is, one sign for each root in $\Delta^\re_+-Y$).
A unique  $\rho\in\mathcal{R}'$ is determined by defining the values $r^{HB}_j=(-1)^{\rho(\a^{HB}_j)}$ inductively as
\begin{align*}
r^{LL}_0&=1,& r^{SL}_0&=1,& r^{SL}_1&=1,\\
r^{LB}_{b+k}&=r^{SB}_{2k+1}s^{LB}_{b+k} ,& 
r^{SB'}_{2k}&=r^{LB}_{b+k}s^{SB'}_{2k},& 
r^{SB'}_{2k+1}&=r^{LB}_{b+k}s^{SB'}_{2k+1},
\end{align*}
for $B\in\{U,L\}$,   $k>0$.
Hence $\phi$ restricted to $\mathcal{R}'$ is one-to-one, and so $\vspan(X)=\ker(\phi)$.
This choice of $\rho$ gives us free choice of the extraspecial signs 
\begin{align*}
\bar{s}_{\a^{SB}_0,\a^{SB}_{2k+1}}&=s^{LB}_{k+b},&
\bar{s}_{\a^{SB}_0,\a^{LB'}_k}&=-s^{SB'}_{2(k+b)} &
\bar{s}_{\a^{SB}_0,\a^{LB'}_k}&= s^{SB'}_{2(k+b)+1}
\end{align*}
\end{proof}
Theorem~\ref{T-sc}\ref{T-sc-H41} is an immediate corollary of this result.

\newpage

\begin{appendix}

\section{Explicit commutator relations}
We can now give all commutator relations involving only real root vectors in rank 2 Kac-Moody algebras. Recall that we write $x_j^{LU}$ instead of $x_{\a_j^{LU}}$, and so on. Given $B\in\{U,L\}$ we define
$$b=\begin{cases} 0&\text{for $B=U$,}\\1&\text{for $B=L$,}\end{cases}\qquad 
B'=\begin{cases} L&\text{for $B=U$,}\\U&\text{for $B=L$.}\end{cases}$$

\begin{corollary} Let $A=H(a,b)$ with $a,b>1$, $ab\geq 4$. The relations in $\mathfrak{g}=\mathfrak{g}(A)$ among pairs of 
real root vectors whose sum is neither imaginary nor zero are:
$$[x_j^{SB},x_j^{SB}]=[x_j^{SB},x_j^{LB}]=[x_j^{LB},x_j^{SB}]=[x_j^{LB},x_j^{LB}]=0,$$
for all $j\in\Z$ and $B\in\{U,L\}$.
\end{corollary}

\begin{proposition}\label{P-allsc-Ha1}  Let $A=H(a,1)$ with $a>4$. 
The relations in $\mathfrak{g}=\mathfrak{g}(A)$ among pairs of 
real root vectors whose sum is neither imaginary nor zero are 
\begin{align*}
[x^{SU}_j,x^{SU}_{j+1}] &= 
  \begin{cases} s^U_jax^{LU}_j & \text{for $j\ge0$,}\\ -s^L_{-j-1}ax^{LU}_{j}& \text{for $j<0$}; \end{cases}&
[x^{SL}_j,x^{SL}_{j+1}]  &=
  \begin{cases} s^L_jax^{LL}_j & \text{for $j\ge0$,}\\ -s^U_{-j-1}ax^{LL}_{j}& \text{for $j<0$}; \end{cases}\\
[x^{SU}_j,x^{LL}_{-j-1}]  &=
  \begin{cases} s^U_jx^{SL}_{-j-2} & \text{for $j\ge0$,}\\ -s^L_{-j-1}x^{SL}_{-j-2}& \text{for $j<0$}; \end{cases}&
[x^{SL}_j,x^{LU}_{-j-1}]  &=\begin{cases} s^U_jx^{SU}_j & \text{for $j\ge0$,}\\ -s^L_{-j-1}x^{SU}_{j}& \text{for $j<0$}; \end{cases}
\end{align*}
\begin{align*}
[x^{SH}_j,x^{SH}_{k}] &= 0 \qquad \text{for $j,k\in\Z$, $|j-k|>1$, $H\in\{L,U\}$};\\
[x^{LH}_j,x^{LH}_{k}] =[x^{SH}_j,x^{LH}_{k}] &= 0\qquad  \text{for $j,k\in\Z$, $H\in\{L,U\}$;}
\end{align*}
together with the relations obtained by swapping the order of these commutators.
\end{proposition}

\begin{proposition}\label{P-allsc-H41}  Let $A=H(4,1)$. 
The relations in $\mathfrak{g}=\mathfrak{g}(A)$ among pairs of 
real root vectors whose sum is neither imaginary nor zero are 
\begin{align*}
[x^{SB}_{2j},x^{SB}_{2k+1}] &= 
  \begin{cases} r^{SB}_{2j}r^{SB}_{2k+1}r^{LB}_{b+j+k}4x^{LB}_{b+j+k} & \text{for $j, k\ge 0$,}\\
     r^{SB}_{-2j+1}r^{SB}_{2k+1}r^{LB}_{b+j+k}4x^{LB}_{b+j+k}& \text{for $0<-j\le k$},\\ 
     r^{SB}_{2j}r^{SB}_{-2k}r^{LB}_{1-b-j-k}4x^{LB}_{b+j+k}& \text{for $0\le j < -k$,}\\
     r^{SB}_{-2j+1}r^{SB}_{-2k}r^{LB}_{1-b-j-k}4x^{LB}_{b+j+k}& \text{for $0< -j < -k$;}  \end{cases}\\
[x^{SB'}_{2i},x^{LB}_{j}]  &=
  \begin{cases} -r^{SB'}_{2j}r^{LB}_{k}r^{SB}_{2(b+k-j)}  x^{SB}_{2(b+k-j)} & \text{for $j,k\ge0$,}\\ 
     -r^{SB'}_{-2j+1}r^{LB}_{k}r^{SB}_{2(b+k-j)}  x^{SB}_{2(b+k-j)} & \text{for $0<-j\le k$},\\ 
     -r^{SB'}_{2j}r^{LB}_{1-k}r^{SB}_{2(b+k-j)}  x^{SB}_{1-2(b+k-j)} & \text{for $0\le j < -k$,}\\
     -r^{SB'}_{-2j+1}r^{LB}_{1-k}r^{SB}_{2(b+k-j)}  x^{SB}_{1-2(b+k-j)} & \text{for $0< -j < -k$;}  \end{cases}
\end{align*}
together with the relations obtained by swapping the order of these commutators; all other such commutators are equal to 0.
Here
$$t_k=\prod_{i=1}^{k}s^{LL}_i s^{LU}_{i-1} s^{SL}_{2i-1} s^{SU}_{2i-1},$$
\begin{align*}
r^{LL}_k&=t_k, &r^{LU}_k&=t_ks^{LU}_ks^{SU}_{2k+1},\\
r^{SL}_{2k} &= t_ks^{LU}_ks^{SU}_{2k+1}s^{SL}_{2k},  & r^{SU}_{2k} &= t_{k+1}s^{SU}_{2k},\\
r^{SL}_{2k+1} &= t_ks^{LU}_ks^{SU}_{2k+1}s^{SL}_{2k+1},  & r^{SU}_{2k+1} &= t_{k+1}s^{SU}_{2k+1}.
\end{align*}
\end{proposition}

\end{appendix}

\end{document}